\documentclass[reqno]{amsart}
\usepackage{color}
\usepackage{amsmath,amssymb,amsfonts,amsthm,bm}
\usepackage{mathrsfs}
\usepackage{mathtools}
\usepackage{graphicx}
\usepackage{enumerate}
\usepackage[numbers, sort&compress]{natbib}
\usepackage[shortlabels]{enumitem}
 \usepackage{newtxtext}
   \usepackage[varvw]{newtxmath}
 \usepackage{textcomp}
\usepackage[reversemp,paperwidth=155mm,paperheight=235mm,top={22mm},headheight={5.5pt},headsep={6.0mm},text={123mm,195mm},marginparsep=5mm,marginparwidth=12mm, bindingoffset=6mm, footskip=10mm]{geometry}
 \newtheorem{theorem}{Theorem}[section]
\newtheorem{lemma}[theorem]{Lemma}

\newtheorem{proposition}[theorem]{Proposition}
\newtheorem{definition}[theorem]{Definition}
 \theoremstyle{definition}
\newtheorem{remark}[theorem]{Remark}

\numberwithin{equation}{section}

\newcommand{\eps}{\varepsilon}
\newcommand{\norm}[1]{\Vert#1\Vert}
\newcommand{\abs}[1]{\left\vert#1\right\vert}

\newcommand{\inner}[1]{\left(#1\right)}
\newcommand{\comi}[1]{\left<#1\right>}

\newcommand{\normm}[1]{{ \vert\kern-0.25ex \vert\kern-0.25ex \vert #1
		\vert\kern-0.25ex \vert\kern-0.25ex \vert}}

 \makeatletter

 \newbox \abstractbox
\renewenvironment{abstract}{\global\setbox\abstractbox=\vbox\bgroup
 \hsize=\textwidth
  \vskip 1.2cm
  \noindent\unskip \textbf{Abstract.}
 }
 {%\vskip12pt \hrule
 \egroup}

\@namedef{subjclassname@2020}{%
	\textup{2020} Mathematics Subject Classification}

\def\@settitle{%
  \bgroup
  \centering
  \vglue1cm
  \fontsize{12}{15}\fontseries{b}\selectfont
  \@title
  \vskip 20pt plus 6pt minus 8pt
  \egroup
}

\def\@setauthors{%
  \begingroup
  \trivlist
  \centering %\bfseries
 \normalsize\@topsep30\p@\relax
  \advance\@topsep by -\baselineskip
  \item\relax
  \andify\authors
 {\rmfamily\authors}%
  \endtrivlist
  \endgroup
}

\def\@setaddresses{\par
  \nobreak \begingroup
\normalsize
  \def\author##1{\nobreak\addvspace\bigskipamount}%
  \def\\{\unskip, \ignorespaces}%
  \interlinepenalty\@M
  \def\address##1##2{\begingroup
    \par\addvspace\bigskipamount\noindent
    \@ifnotempty{##1}{(\ignorespaces##1\unskip) }%
    {\ignorespaces##2}\par\endgroup}%
  \def\curraddr##1##2{\begingroup
    \@ifnotempty{##2}{\nobreak\indent{\itshape Current address}%
      \@ifnotempty{##1}{, \ignorespaces##1\unskip}\/:\space
      ##2\par}\endgroup}%
  \def\email##1##2{\begingroup
    \@ifnotempty{##2}{\nobreak\noindent{\itshape E-mail address}%
      \@ifnotempty{##1}{, \ignorespaces##1\unskip}\/:
       ##2\par}\endgroup}%
   \def\urladdr##1##2{\begingroup
    \@ifnotempty{##2}{\nobreak\indent{\itshape URL}%
      \@ifnotempty{##1}{, \ignorespaces##1\unskip}\/:\space
      \ttfamily##2\par}\endgroup}%
  \addresses
  \endgroup
}

  \renewcommand\section{\@startsection{section}{1} %
  \z@{.5\linespacing\@plus.7\linespacing}{.5\linespacing}
 {\normalfont\large\bfseries\boldmath}}

  \def\subsection{\@startsection{subsection}{2}%
  \z@{.5\linespacing\@plus.7\linespacing}{.2\linespacing}%
  {\raggedright\normalfont\bfseries\boldmath}}
  
\def\subsubsection{\@startsection{subsubsection}{3}%
  \z@{.5\linespacing\@plus.7\linespacing}{-.5em}%
  {\normalfont\bfseries}}

\makeatother

\begin{document}

\title[Global well-posedness of the Prandtl-Shercliff model]{Global Well-Posedness for the 2D and 3D Prandtl-Shercliff Model}

\author[W.-X.Li, Z.Xu and A. Yang]{Wei-Xi Li, Zhan Xu and Anita Yang}

\address[W.-X. Li]{School of Mathematics and Statistics,   \&  Hubei Key Laboratory of Computational Science, Wuhan University, Wuhan 430072,  China } \email{wei-xi.li@whu.edu.cn}

\address[Z. Xu]{School of Mathematics and Statistics,  Wuhan University, Wuhan 430072, China\newline
and \newline
Department of Applied Mathematics, The Hong Kong Polytechnic
University, Kowloon, Hong Kong, SAR, China.}
\email{xuzhan@whu.edu.cn}

\address[A. Yang]{Wuhan Institute for Math \& AI, Wuhan University, Wuhan 430072,  China}
\email{anitaymt@outlook.com}

\keywords{Prandtl-Shercliff model, Global Well-posedness, analytic regularization effect}

 \subjclass[2020]{76D10,76D03}

\begin{abstract}
We investigate the Prandtl-Shercliff model in both two and three dimensions. For the two-dimensional case, we establish global-in-time well-posedness in Sobolev spaces without any structural assumptions on the initial data. Furthermore, we show that the solution exhibits an analytic regularization effect in all variables, which holds globally in time and in space up to the boundary. For the three-dimensional case, we study a linearized version of the model and prove its global-in-time well-posedness for initial data that are analytic in only one tangential direction.  The proofs rely crucially on the intrinsic non-local diffusion induced by the Shercliff boundary layer.
\end{abstract}
 
\maketitle
 
 %\setcounter{tocdepth}{1}
%\tableofcontents

\section{Introduction and main results}

The Prandtl-Shercliff model is a specific type of boundary layer system for magnetohydrodynamics (MHD) flows, which describes the behavior of an electrically conducting fluid  confined to a thin layer under the influence of a transverse magnetic field. The typical feature of this model is that the magnetic field  creates a distinctive, flat velocity profile known as a Shercliff layer. The model is derived from the full MHD equations by applying a boundary layer approximation for high Hartmann number.  Regarding  the mathematical  formulation of the governing equations that follows, we refer to the work   of G\'erard-Varet and Prestipino \cite{MR3657241} for full details.  
Without loss of generality, we consider the fluid domain to be the half-space in $\mathbb{R}^2$ or $\mathbb{R}^3$,  namely,   
\begin{equation*}
    \mathbb{R}^d_+=\{(x,z)\in\mathbb{R}^d;\ x=(x_1,\cdots,x_{d-1})\in \mathbb{R}^{d-1}, z>0\},\quad d=2 \ \textrm{ or } 3. 
\end{equation*}
The governing equations of the Prandtl-Shercliff model in $\mathbb{R}_+^d$ are then given by
\begin{equation}\label{eq:main}
		\left\{
		\begin{aligned}
                &  (\partial_t+u\cdot\partial_x+ w\partial_z-\partial_z^2)u +\partial_x p=\partial_{x_1} f,\\
			&  \partial_{x_1} u+ \partial_z^2 f = 0,\quad  \partial_x\cdot u+ \partial_z w= 0,\\
			&  (u,w,f)|_{z=0}=(0,0,0),\qquad \lim\limits_{z \to \infty}(u,f)=(U_{\infty}(t,x),F_{\infty}(t,x)),\\
			&  u|_{t=0}=u_0,		\end{aligned}
		\right.
	\end{equation}
where we denote by $\partial_x=(\partial_{x_1}, \cdots, \partial_{x_{d-1}})$ the tangential gradient for $x\in \mathbb{R}^{d-1}.$ The unknowns in \eqref{eq:main} are the velocity field $(u, w)$ and the tangential magnetic field $f = (f_1, \cdots, f_{d-1})$, where $u = (u_1, \cdots, u_{d-1})$ and $w$ denote the tangential and vertical velocity components, respectively.   The functions $U_{\infty}(t,x)$, $F_{\infty}(t,x)$, and $p(t,x)$ in \eqref{eq:main} are given data, representing the boundary values of the tangential velocity, magnetic field, and pressure, respectively, and  satisfying the Bernoulli law:
\begin{equation*}
(\partial_t +U_{\infty}\cdot \partial_x)U_{\infty}+\partial_x p=\partial_{x_1}F_{\infty}.
\end{equation*}
The Prandtl-Shercliff model \eqref{eq:main} combines two fundamental physical effects: the viscous effects of the Prandtl layer and the anisotropic, non-local effects induced by the magnetic field in the Shercliff layer.  To simply the argument we will assume that $(U_{\infty},F_{\infty})\equiv (0,0)$ in   system \eqref{eq:main}. This is without loss of generality, as the result for the general case follows from an analogous argument.  Hence, we consider the reduced system:  
	\begin{equation}\label{eq:main2}
		\left\{
		\begin{aligned}
                &  (\partial_t+u\cdot\partial_x+ w\partial_z-\partial_z^2)u=\partial_{x_1} f,\\
			&  \partial_{x_1} u+ \partial_z^2 f = 0,\quad  \partial_x\cdot u+ \partial_z w= 0,\\
			&  (u,w,f)|_{z=0}=(0,0,0),\qquad \lim\limits_{z \to +\infty}(u,f)=(0,0),\\
			&  u|_{t=0}=u_0.
		\end{aligned}
		\right.
	\end{equation}

Before presenting the main result on the global well-posedness of system \eqref{eq:main2}, we recall some known results concerning the Prandtl-type system. The classical Prandtl system  (without a magnetic field) 
\begin{equation*}\label{prand}
		\left\{
		\begin{aligned}
            &  (\partial_t+u^P\cdot\partial_x+ w^P\partial_z-\partial_z^2)u^P +\partial_x p^P=0,\\
            &\partial_x\cdot u^P+\partial_z w^P=0,\\
			&  (u^P,w^P)|_{z=0}=(0,0),\qquad \lim\limits_{z \to \infty}u^P=U^P_{\infty}(t,x),\\
			& u^P|_{t=0}=u^P_0,
		\end{aligned}
		\right.
	\end{equation*}
can be viewed as a degenerate version of the Navier-Stokes equations lacking tangential diffusion. In this system, there is no independent evolution equation for the normal component 
$w^P$; instead, it is fully determined by the divergence-free condition and the boundary condition:
\begin{equation*}
w^P(t,x,z)=-\int^z_0\partial_x\cdot u^P(t,x,\tilde z)d\tilde z.
\end{equation*}
It is the non-local term $w^P$  that leads to a loss of tangential derivatives, which is the major difficulty  in establishing the well-posedness of the Prandtl system. In the absence of Oleinik's monotonicity condition, the  Prandtl system is usually ill-posed in Sobolev spaces, as shown in \cite{MR2952715, MR1476316, MR2601044} and   references therein. So far, the well-posedness property has been extensively studied in a variety of function spaces.  Here we only mention  the recent works of \cite{MR3327535,MR3385340,MR1697762,MR2020656} for Sobolev spaces,  \cite{MR2975371,MR1617542,MR4271962,MR3461362} for analytic spaces  and \cite{MR3429469,MR4055987,MR4465902,MR3925144,MR4701733} for more general Gevrey-class spaces. 

Compared with the classical Prandtl system, the distinctive feature of system \eqref{eq:main2} is the presence of the non-local Shercliff term, $\partial_{x_1}f$, which arises from rapid velocity diffusion along magnetic field lines. 
Analogous to the classical Prandtl system,  the loss of tangential derivatives also occurs in the non-local term $w.$  Nevertheless,  
the local-in-time Sobolev well-posedness of the two-dimensional system \eqref{eq:main2} without any structural assumptions, established by \cite{prestipino2018,MR3657241}, suggests that this non-local Shercliff term may suppress flow instabilities through tangential diffusion.  

This work aims to establish the corresponding global-in-time theory, after the earlier local-in-time results  \cite{prestipino2018,MR3657241}. Precisely, we first establish the global-in-time Sobolev well-posedness of the two-dimensional system \eqref{eq:main2} by extensively exploiting   the stabilizing effect of the non-local Shercliff term. However, for the three-dimensional system \eqref{eq:main2}, the issue of Sobolev well-posedness remains open, even in the local-in-time setting. In this work, we address a linearized version of system \eqref{eq:main2} and establish the global well-posedness for initial data that are analytic in only one tangential component. 
Furthermore, we prove that these solutions exhibit a space-time analytic smoothing effect, analogous to the one observed in the heat equation, which holds globally in time and in space up to the boundary.

Due to the strong diffusion inherent in the heat equation, the associated analytic regularization effect and the radius of analyticity have been extensively studied. Such parabolic regularization was  established for the Navier-Stokes equations on the whole space or torus by Foias and Temam \cite{MR1026858}, who proved space-time analyticity via $L^2$-energy estimates and Fourier techniques. Since then, this Fourier-based approach and  subsequently developed  more modern analytic methods beyond $L^2$ have been applied to study analyticity for the Navier-Stokes equations and more general parabolic equations in various function spaces. While the aforementioned results mainly concern the classical Navier-Stokes equations, much less is known about the analytic regularity of  the Prandtl-type equations. Unlike the heat or Navier-Stokes equations, the Prandtl-type equations are degenerate parabolic equations. This degeneracy, often manifesting as a lack of diffusion in one or more tangential directions, means that one can generally only expect the propagation of initial regularity rather than smoothing in those directions. In the two-dimensional case, the Oleinik's monotonicity condition yields an intrinsic hypoelliptic structure in the tangential direction for the  Prandtl equation, which in turn leads to Gevrey-class regularity at positive times, even for initial data of finite Sobolev regularity (see \cite{MR3493958}). On the other hand, for general initial data without  structural assumption, one may assume strong analyticity in the degenerate directions in order to establish the analytic smoothing effect in other directions, as shown in \cite{preprint-liyangzhang} for the 2D and 3D Prandtl equations. 

\subsection{Notations}  
Before stating the main results, we first introduce some notations that will be used throughout this paper. 

\begin{enumerate}[label={(\arabic*)}, leftmargin=*, widest=ii]
\item	Let $d=2$ or $3$, we will use $\|\cdot\|_{L^2}$ and $(\cdot,\cdot)_{L^2}$ to denote the norm and inner product of $L^2=L^2(\mathbb R_+^d)$, and use the notations $\|\cdot\|_{L^2_x}$ and $(\cdot,\cdot)_{L^2_x}$ when the variable $x$ is specified. Similar notations will be used for $L^{\infty}$. In addition, we use $H^p_xH^q_z=H_x^p(\mathbb{R}^{d-1};H_z^q(\mathbb{{R}}_{+}))$ for the classical Sobolev space.  Similarly,  $H^p_x L^q_{z} = H_x^p(\mathbb R^{d-1}; L_z^q(\mathbb {R_+})).$

\item   For a given norm $\norm{\cdot}$ and a given vector-valued function $\mathbf{A}=(A_1,\cdots,A_n)$, we define
 \begin{equation*}
 	\norm{\mathbf{A}}\stackrel{\rm def}{=}\Big(\sum_{1\leq j\leq k} \norm{A_j}^2\Big)^{\frac12}.
 \end{equation*} 

 \item The symbols $\alpha $ and $\beta$ denote  multi-indices in either $\mathbb{Z}_+^2$ or $\mathbb{Z}_+^3$, depending on context. For a given multi-index $\alpha = (\alpha_1, \dots, \alpha_n) \in \mathbb{Z}_+^n$ ($n=2$ or $3$), we define
\begin{equation*}
 	\tilde\alpha\stackrel{\rm def}{=}\alpha-(1,0,\cdots,0)=(\alpha_1-1,\alpha_2, \cdots,\alpha_n)\in\mathbb Z_+^n 
 \end{equation*} 
for $\alpha_1\ge 1$, and
  \begin{equation*}
 	\alpha_*\stackrel{\rm def}{=}\alpha-(0,0,\cdots,1)=(\alpha_1,\alpha_2, \cdots,\alpha_n-1)\in\mathbb Z_+^n
 \end{equation*} 
for $\alpha_n\ge 1$.
\end{enumerate}

\subsection{Function spaces and main results}

We now state the main results for the two- and three-dimensional cases. In  two dimensions, system \eqref{eq:main2} takes the form 
\begin{equation}\label{eq:2D}
		\left\{
		\begin{aligned}
                &  \partial_t u+u\partial_x u+ w\partial_z u - \partial_z^2u=\partial_x f,\\
			&  \partial_{x}u+ \partial_z^2 f = 0,\quad  \partial_x u+ \partial_zw = 0,\\
			&  (u,w,f)|_{z=0}=(0,0,0),\qquad \lim\limits_{z \to \infty}(u,f)=(0,0),\\
			&  u|_{t=0}=u_0,
		\end{aligned}
		\right.
\end{equation}
which is  posed on $\mathbb R_+^2=\big\{(x,z); \,  x\in\mathbb R, z>0\big\}.$
We work with the anisotropic weighted Sobolev space  $\mathcal{H}^1(\mathbb R_+^2)$, defined by 
\begin{equation}\label{def:weightspaceH1}
 \mathcal{H}^1= \mathcal{H}^1(\mathbb R_+^2)\stackrel{\rm def}{=}\Big\{h(x,z): \mathbb R_+^2\rightarrow\mathbb R;\ \norm{h}_{\mathcal{H}^1} <+\infty\Big\},  
\end{equation}
with   the norm $\norm{\cdot}_{\mathcal{H}^1}$ defined by
\begin{equation}\label{def:weightnorm}
	\norm{h}_{\mathcal{H}^1}^2\stackrel{\rm def}{=} \| h\|^2_{H_x^1L_z^2}+\|\comi z\partial_zh\|^2_{H_x^1L_z^2},
\end{equation}
where, here and below, $\comi z \stackrel{\rm def}{=} (1+z^2)^{\frac{1}{2}}.$  The corresponding inner product  is defined as
\begin{equation*}
	\inner{g,\ h}_{\mathcal{H}^1}\stackrel{\rm def}{=}\big(g,\ h\big)_{H_x^1L_z^2}+\big(\comi z\partial_z g,\  \comi z\partial_z h\big)_{H_x^1L_z^2}.
\end{equation*}

\begin{theorem}\label{thm:2D}
Assume the initial data ${u}_0\in\mathcal{H}^1(\mathbb R_+^2)$, compatible with the boundary condition in system \eqref{eq:2D}. Then there exists a small constant $\varepsilon_0>0$ such that if
\begin{equation*}%\label{sma}
\|u_0\|_{\mathcal{H}^1}\leq \varepsilon_0,
\end{equation*}
then the two-dimensional Prandtl-Shercliff system \eqref{eq:2D} admits a unique global-in-time solution $u\in L^\infty([0,+\infty[;\mathcal{H}^1)$ satisfying that 
\begin{equation*}
\forall\ t\ge0,\quad \|u(t)\|_{\mathcal{H}^1} \leq \varepsilon_0.
\end{equation*}
Moreover, the solution $u$ is space-time analytic at any positive time, satisfying that   
\begin{align}\label{result:smoothing}
\forall \ k,m,j\ge0, \quad \sup_{t\ge 0}t^{k+m+\frac{j}{2}}\norm{\partial_t^k\partial_x^m\partial_z^ju}_{\mathcal{H}^1}\leq \varepsilon_0C_0^{k+m+j}(k+m+j)!  
\end{align}
for some constant $C_0>0$. 
\end{theorem}

In the three-dimensional case,   
  we denote the velocity field by $(u,v,w)$ and   the tangential magnetic field by $(f,g)$, with spatial variables $(x,y,z)\in\mathbb R_+^3$.  We consider the following linearization of system \eqref{eq:main2} around a given shear flow $(U,V)$:
	\begin{equation}\label{3Dlinear}
		\left\{
		\begin{aligned}
            &  \partial_t u+U\partial_x u+ V\partial_y u +w\partial_zU- \partial_z^2u=\partial_x f,\\
            &  \partial_t v+U\partial_x v+ V\partial_y v +w\partial_zV- \partial_z^2v=\partial_x g,\\
			&  \partial_x u+ \partial_z^2 f = 0, \qquad\partial_x v+ \partial_z^2 g = 0,\\   
			&  \partial_x u + \partial_y v +\partial_z w= 0,\\
			&  (u,v,w,f,g)|_{z=0}=(0,0,0,0,0),\qquad \lim\limits_{z \to + \infty}(u,v,f,g)=(0,0,0,0),\\
			&  (u,v)|_{t=0}=(u_0,v_0),
		\end{aligned}
		\right.
	\end{equation} 
where the shear profiles  $U=U(t,z)$ and $V=V(t,z)$ satisfy the heat equations
	\begin{equation}\label{Heateq}
		\left\{
		\begin{aligned}
&\partial_tU-\partial_z^2U=0,\quad\partial_tV-\partial_z^2V=0,\\
&(U,V)|_{t=0}=(U_0,V_0),\quad (U,V)|_{z=0}=(0,0).
		\end{aligned}
		\right.
	\end{equation}
With the non-negative weight function $\mu_\lambda$ defined  as
\begin{equation*}%\label{mulambda}
	\mu_{\lambda}=\mu_{\lambda}(t,z)\stackrel{\rm def}{=}\exp\Big(\frac{\lambda z^2}{4(1+t)}\Big), \quad 0\leq \lambda\leq 1,
\end{equation*}
we associate a weighted Lebesgue space $L_{\mu_\lambda}^2(\mathbb R_+)$  by setting 
 \begin{equation}\label{leb}
 	L_{\mu_\lambda}^2(\mathbb R_+)\stackrel{\rm def}{=}\Big\{h (z): \mathbb{R}_+\to\mathbb{R}; \ \   \norm{h}_{L_{\mu_\lambda}^2}\stackrel{\rm def}{=}\Big(\int_{\mathbb{R}_+} \mu_\lambda(z) h(z)^2dz\Big)^{1\over2}<+\infty\Big\}.
 \end{equation}
 More generally,   define the weighted Sobolev space 
\begin{align*}
H^{m}_{\mu_\lambda}(\mathbb R_+)\stackrel{\rm def}{=}\Big\{h (z): \mathbb{R}_{+}\to\mathbb{R}; \  \  \norm{h}_{H^{m}_{\mu_\lambda}}\stackrel{\rm def}{=}\Big(\sum^m_{i=0}\norm{\partial_z^{i}h}^2_{L^2_{\mu_\lambda}}\Big)^{1\over2}<+\infty\Big\}.\end{align*}
In particular, when $\lambda = 1$, we denote $\mu = \mu_1$, that is,
\begin{equation*}%\label{defmu}
 \mu=\mu(t,z)\stackrel{\rm def}{=}\exp\Big(\frac{z^2}{4(1+t)}\Big).
\end{equation*}

\begin{proposition}\label{thm:heat}
Let the weighted Sobolev space $H^m_{\mu_\lambda}(\mathbb R_+)$  be defined as above.   Assume the initial data of system \eqref{Heateq} satisfy the compatibility condition and the bound  
\begin{align*}%\label{ass}
  \norm{(U_0,V_0)}_{H^3_{\mu_{in}}}\leq \varepsilon_1, 
\end{align*}
where $\eps_1>0$ is a constant and 
 \begin{equation*}
 	\mu_{in}\stackrel{\rm def}{=}\mu(0,z)=\exp\Big (\frac{z^2}{4}\Big).
 \end{equation*}
 Let $(U,V) \in L^\infty\big([0,+\infty[\ ; \  H^{3}_{\mu}\big)$ be the corresponding solution to the heat equations \eqref{Heateq}.  
If, in addition, the initial data satisfy
 \begin{align*}\label{compati}
 \int^{+\infty}_0 z U_0(z)dz= \int^{+\infty}_0 z V_0(z)dz=0,
\end{align*}
then there exists a constant $C_1>0$ such that   
\begin{equation}\label{EstUV}
\forall\ t\ge0,\quad\norm{\partial_z(U,V)}_{L^{\infty}_z}+\norm{z\partial_z(U,V)}_{L^{\infty}_z}+\norm{z\partial_z^2(U,V)}_{L^{\infty}_z} 
\leq C_1\eps_1 (1+t)^{-\frac{3}{2}}.
\end{equation}
\end{proposition}

\begin{remark}
The decay rate of $(1+t)^{-\frac{3}{2}}$	  in \eqref{EstUV} is not sharp; indeed, Proposition \ref{thm:prioriheat}  provides a refined rate of $ (1+t)^{-\frac{8-\delta}{4}}$ for any $0<\delta<2$. 
\end{remark}

\begin{remark}
Owing to the classical analytic smoothing effect of the heat equation, the solutions $U$ and $V$ in Proposition \ref{thm:heat} instantaneously become space-time analytic for all $t > 0$. Specifically, for any $t > 0$ and any integers $k, j \ge 0$, they satisfy the estimate
  \begin{multline}\label{smforheta}
  t^{k+\frac{j}{2}}\Big(\norm{\partial_t^k\partial_z^{j+1}(U,V)}_{L^{\infty}_z} +\norm{z\partial_t^k\partial_z^{j+1}(U,V)}_{L^{\infty}_z}  
  + \norm{z\partial_t^k\partial_z^{j+2}(U,V)}_{L^{\infty}_z}\Big)\\ 
\leq \eps_1 C_1^{k+j} (1+t)^{-\frac{3}{2}}(k+j)!.
\end{multline}
This estimate can be derived by combining the inductive argument and the proof of \eqref{EstUV}. 
\end{remark}

\begin{definition}\label{defspace}
Let $\rho>0$, the analytic function space $\mathcal{X}_{\rho}(\mathbb{{R}}_+^3)$ consists of all smooth functions $h$ which are analytic in the tangential variable $y$ and satisfy  $\|h\|_{\mathcal{X}_{\rho}} < +\infty$, with
\begin{equation*}
\|h\|^2_{\mathcal{X}_{\rho}}\stackrel{\rm def}{=}\sum_{m=0}^{+\infty}L_{\rho,m}^2\Big(\|\partial_{y}^mh\|^2_{L^2}+\|\partial_{y}^m\partial_z h\|^2_{L^2}\Big),
\end{equation*}
where, here and below,
\begin{equation}\label{lrm}
L_{\rho,m}\stackrel{\rm def}{=}\frac{\rho^{m+1}}{m!},\quad m\ge0,\ \rho >0.
\end{equation}
\end{definition}

\begin{theorem}\label{thm:linear}
Assume the coefficients $U$ and $V$ of the three-dimensional linearized Prandtl-Shercliff system \eqref{3Dlinear} satisfy the decay estimate \eqref{EstUV},  and suppose  the initial data   $u_0, v_0\in \mathcal{X}_{\rho_0}( \mathbb{{R}}_+^3)$ for some $\rho_0 >0,$ compatible to the boundary condition in \eqref{3Dlinear}.
If the constant $\eps_1$ in \eqref{EstUV} is sufficiently small, then system \eqref{3Dlinear} admits a unique global-in-time solution $(u,v)\in L^{\infty}([0,+\infty[;\mathcal{X}_{\rho})$ satisfying  
\begin{equation*}
\forall\ t\ge0,\quad \|(u,v)(t)\|_{\mathcal{X}_{\rho}}\leq \|(u_0,v_0)\|_{\mathcal{X}_{\rho_0}},
\end{equation*}
where
\begin{equation}\label{3Ddefrho}
\rho=\rho(t)\stackrel{\rm def}{=}\frac{\rho_0}{2}+\frac{\rho_0}{2}(1+t)^{-\frac12}.
\end{equation}
Moreover, there exists a constant $C_*>0$ such that
\begin{equation}\label{analy}
	\forall \ k,m,j\ge0, \quad \sup_{t\ge 0}t^{k+m+\frac{j}{2}}\norm{\partial_t^k\partial_x^m\partial_z^j(u,v)}_{\mathcal X_\rho}\leq C_*^{k+m+j+1}(k+m+j)!. 
\end{equation}
\end{theorem}

\begin{remark}
The analyticity estimates \eqref{result:smoothing} and \eqref{analy}  hold globally in time  and persist up to the boundary. Establishing analyticity in domains with boundaries is usually  nontrivial,   since the Fourier-based approach is no longer applicable and one must carefully handle non-vanishing boundary terms.
\end{remark}

\begin{remark}
The analyticity radius may help to understand the turbulence in fluid dynamics (cf. \cite{MR1467006,MR1331063,  MR1073624, MR3215083} for instance).  The analyticity estimates \eqref{result:smoothing} and \eqref{analy}   yield  that  the analyticity radius   in $(t,x)$  is  bounded below by a constant multiple of $t,$ while in the $z$-direction the radius remains bounded below by a constant multiple of $\sqrt{t}$.  The anisotropic radii of analyticity reflect the underlying anisotropic diffusion. Specifically, the Shercliff term generates a non-local diffusion that behaves like the fractional Laplacian $(-\Delta_x)^{\frac{1}{2}}.$
\end{remark}
  
\begin{remark}
 In the three-dimensional case, although the Shercliff term may provide dissipation along one tangential direction, it remains unclear whether the results of Theorem \ref{thm:linear} extend to the nonlinear setting. The main difficulty lies in selecting a suitable weight function. To the best of our knowledge, even in the analytic setting, the global well-posedness of the 3D nonlinear system remains  open.
\end{remark}
  
\begin{remark}\label{anasmooth}
Given the validity of estimate \eqref{smforheta},  the proof of \eqref{analy} is directly analogous to that of \eqref{result:smoothing}  with no additional difficulties. We therefore omit the details here.
\end{remark}

The paper is organized as follows. Sections \ref{sec:2D} and \ref{sec:3D} are devoted to the proofs of Theorems \ref{thm:2D} and \ref{thm:linear}, respectively. Appendix \ref{sec:ineq} contains the proofs of some straightforward inequalities.

To simplify notation, throughout this paper we use the capital letter $C\ge 1$ to denote a generic positive constant that may vary from line to line. This constant depends on the Sobolev embedding constants, but is independent of  any other parameters specified in the proof.

\section{Proof of Theorem \ref{thm:2D}}\label{sec:2D}
This section is devoted to proving Theorem \ref{thm:2D}. Specifically, through the two subsections, we establish in turn the global well-posedness of system \eqref{eq:2D} and the analytic smoothing effect in all variables, thus completing the proof of Theorem \ref{thm:2D}.

\subsection{Global existence and uniqueness of system \eqref{eq:2D}}

We establish in this part the global-in-time existence and uniqueness of system \eqref{eq:2D} in the Sobolev setting. To address this, it suffices to 
 derive an   \emph{a priori} energy estimate for system \eqref{eq:2D}.   The global-in-time existence and uniqueness then follow by a standard regularization  argument. Hence,  for brevity, we only present the proof of  the following \emph{a priori} estimate  and omit the  regularization procedure.

 \begin{theorem}[\emph{A priori} estimate]\label{thm:wellposedness}
Let $\mathcal{H}^1(\mathbb R^2_+)$ be the anisotropic weighted Sobolev space  as defined in \eqref{def:weightspaceH1}. Assume the initial datum ${u}_0\in\mathcal{H}^1(\mathbb R_+^2)$, compatible with the boundary condition in system \eqref{eq:2D}. Then there exists a small constant $\varepsilon_0>0$ such that  if $u\in L^\infty([0,+\infty[;\mathcal{H}^1)$ is a global solution to system \eqref{eq:2D} and the initial datum $u_0$ satisfies
 \begin{equation}\label{ass:initial}
 \norm{u_0}_{\mathcal{H}^1}\leq \varepsilon_0,
 \end{equation}
 then
 \begin{equation}\label{ret:pri}
  \forall\ t\ge 0,\quad \norm{u(t)}_{\mathcal{H}^1}^2+\int^t_0\mathcal{D}(s)ds\leq  \varepsilon_0^2,
 \end{equation}
 where here and below,
     \begin{equation}\label{def:energy}
\mathcal{D}(t)\stackrel{\rm def}{=}\norm{\comi z \partial_z^2u(t)}_{H_x^1L_z^2}^2+\norm{\comi z\partial_xu(t)}_{H_x^1L_z^2}^2.
    \end{equation}
 \end{theorem}

\begin{proof}
It suffices to prove the following estimate:
\begin{equation}\label{est:pri}
\forall\ t\ge0,\quad \frac{1}{2}\frac{d}{dt}\norm{u(t)}_{\mathcal{H}^1}^2+\mathcal{D}(t) \leq C\norm{u(t)}_{\mathcal{H}^1}\mathcal{D}(t).  
\end{equation}
Assuming \eqref{est:pri} holds, a standard bootstrap argument yields assertion \eqref{ret:pri}. To see this, suppose the solution satisfies
 \begin{equation}\label{ass:pri}
  \forall\ t\ge 0,\quad   \norm{u(t)}_{\mathcal{H}^1}^2+\int^t_0\mathcal{D}(s)ds\leq  2\varepsilon_0^2.
 \end{equation}
 This with \eqref{est:pri} yields
\begin{equation*}%\label{dddd}
\forall\ t\ge0,\quad \frac{1}{2}\frac{d}{dt}\norm{u(t)}_{\mathcal{H}^1}^2+(1-\sqrt{2}C \varepsilon_0)\mathcal{D}(t) \leq 0. 
\end{equation*}
Choosing $\varepsilon_0>0$ small enough such that $1-\sqrt{2}C\varepsilon_0 \geq \frac{1}{2}$, we obtain from the above estimate  that
\begin{equation*}
\forall\ t\ge0,\quad \frac{d}{dt}\norm{u(t)}_{\mathcal{H}^1}^2+\mathcal{D}(t) \leq 0.
\end{equation*}
Integrating this in time and applying the  initial assumption \eqref{ass:initial}, we close the bootstrap argument and obtain the desired estimate \eqref{ret:pri}.

It remains to prove the key estimate \eqref{est:pri}. Recall the norm $\norm{\cdot}_{\mathcal{H}^1}$ is defined in \eqref{def:weightnorm}. Taking the $\mathcal{H}^1$-product with $u$ on both sides of the velocity equation in system \eqref{eq:2D} yields
\begin{equation}\label{we}
    \frac{1}{2}\frac{d}{dt}\norm{u}_{\mathcal{H}^1}^2-\big(\partial_z^2u,\ u\big)_{\mathcal{H}^1}=\inner{\partial_xf,\ u}_{\mathcal{H}^1}-\inner{u\partial_xu+w\partial_zu,\ u}_{\mathcal{H}^1}.
\end{equation}
Using integration by parts and observing $\partial_z^2u|_{z=0}=u|_{z=0}=0$, we obtain
\begin{multline*}
-\big(\partial_z^2u,\ u\big)_{\mathcal{H}^1}=-\big(\partial_z^2u,\ u\big)_{H_x^1L_z^2}-\big(\comi z\partial_z^3u,\ \comi z\partial_zu\big)_{H_x^1L_z^2}\\
=\norm{\partial_zu}_{H_x^1L_z^2}^2+\norm{\comi z\partial_z^2u}_{H_x^1L_z^2}^2+2\big(z\partial_z^2u,\ \partial_zu\big)_{H_x^1L_z^2}=\norm{\comi z\partial_z^2u}_{H_x^1L_z^2}^2.
\end{multline*}
Similarly, using the second equation $\partial_xu+\partial_z^2f=0$ in system \eqref{eq:2D} and the boundary condition  $\partial_z^2f|_{z=0}=f|_{z=0}=0$,  we find   
\begin{align*}
 &\inner{\partial_xf,\ u}_{\mathcal{H}^1}= \inner{\partial_xf,\ u}_{H_x^1L_z^2}+\inner{\comi z\partial_x\partial_zf,\ \comi z\partial_zu}_{H_x^1L_z^2}\\
& = -\inner{f,\ \partial_xu}_{H_x^1L_z^2}-\inner{\comi z\partial_zf,\ \comi z\partial_x\partial_zu}_{H_x^1L_z^2} \\
& = \big(f,  \partial_z^2f\big)_{H_x^1L_z^2}+\big(\comi z\partial_zf,  \comi z\partial_z^3f\big)_{H_x^1L_z^2}
 =-\norm{\comi z\partial_z^2f}_{H_x^1L_z^2}^2=-\norm{\comi z\partial_xu}_{H_x^1L_z^2}^2.
\end{align*}
Substituting the two estimates above into \eqref{we} and using the definition \eqref{def:energy} of $\mathcal{D}$, we get
\begin{equation}\label{wee}
      \frac{1}{2}\frac{d}{dt}\norm{u}_{\mathcal{H}^1}^2+\mathcal{D}=-\inner{u\partial_xu+w\partial_zu,\ u}_{\mathcal{H}^1}.  
\end{equation}
It remains to handle the right-hand side of \eqref{wee}. Recalling definition \eqref{def:weightnorm} of the norm $\norm{\cdot}_{\mathcal{H}^1}$, we write
\begin{multline}\label{eses}
    -\inner{u\partial_xu+w\partial_zu,\ u}_{\mathcal{H}^1}
=-\inner{u\partial_xu+w\partial_zu,\ u}_{H_x^1L_z^2} \\ -\inner{\comi z\partial_z(u\partial_xu+w\partial_zu),\ \comi z\partial_zu}_{H_x^1L_z^2}.
\end{multline}
For the first term on the right-hand side of \eqref{eses}, the Sobolev inequality gives  
\begin{align*}
 &\big|\inner{u\partial_xu+w\partial_zu,\ u}_{H_x^1L_z^2}\big|\\
 &\leq C\norm{\comi z^{-1}u}_{H_x^1L_z^\infty}\norm{\comi z\partial_xu}_{H_x^1L_z^2}\norm{u}_{H_x^1L_z^2}+C\norm{w}_{H_x^1L_z^\infty}\norm{\partial_zu}_{H_x^1L_z^2}\norm{u}_{H_x^1L_z^2}\\
& \leq C\norm{\partial_zu}_{H_x^1L_z^2}\norm{\comi z\partial_xu}_{H_x^1L_z^2}\norm{u}_{H_x^1L_z^2}\leq C\norm{\comi z\partial_z^2u}_{H_x^1L_z^2}\norm{\comi z\partial_xu}_{H_x^1L_z^2}\norm{u}_{\mathcal{H}^1},
\end{align*}
where the last line uses Hardy's inequality as well as the definition of $\norm{\cdot}_{\mathcal{H}^1}$ (see \eqref{def:weightnorm}).
For the second term,  using integration by parts and Hardy's inequality, we obtain
\begin{align*}
&  \big|\inner{\comi z\partial_z(u\partial_xu+w\partial_zu),\ \comi z\partial_zu}_{H_x^1L_z^2}\big|\\
  & \leq  C\norm{\comi z(u\partial_xu+w\partial_zu)}_{H_x^1L_z^2}\norm{\comi z\partial_z^2u}_{H_x^1L_z^2}\\
 & \leq C\big(\norm{u}_{H_x^1H_z^1}+\norm{\comi z\partial_zu}_{H_x^1L_z^2}\big)\norm{\comi z\partial_xu}_{H_x^1L_z^2}\norm{\comi z\partial_z^2u}_{H_x^1L_z^2}\\
 & \leq C\norm{u}_{\mathcal{H}^1}\norm{\comi z\partial_xu}_{H_x^1L_z^2}\norm{\comi z\partial_z^2u}_{H_x^1L_z^2}.
\end{align*}
Therefore, recalling the definition \eqref{def:energy} of 
$\mathcal D$, we combine the above estimates with \eqref{eses} to obtain
\begin{equation*}%\label{nonlinearterm}
-\inner{u\partial_xu+w\partial_zu,\ u}_{\mathcal{H}^1}
\leq C\norm{u}_{\mathcal{H}^1}\norm{\comi z\partial_xu}_{H_x^1L_z^2}\norm{\comi z\partial_z^2u}_{H_x^1L_z^2}\leq C\norm{u}_{\mathcal{H}^1}\mathcal{D}.    
\end{equation*}
Substituting this estimate  into \eqref{wee}  yields assertion \eqref{est:pri}. The  proof of Theorem \ref{thm:wellposedness} is thus completed. 
\end{proof}

\subsection{Proof of Theorem \ref{thm:2D}: analytic smoothing effect}\label{subsec:y}

This subsection is devoted to establishing the analytic smoothing effect in all variables. To do this, we first introduce two auxiliary norms as follows.

\begin{definition}\label{def:energyy}
Let $0<r<1$ be a parameter to be chosen later,  and let the norm $\norm{\cdot}_{\mathcal{H}^1}$ be defined as in \eqref{def:weightnorm}. We define  two auxiliary norms $\abs{\cdot}_{X_{r}}$ and $\abs{\cdot}_{Z_{r}}$  as follows:
 \begin{equation}\label{dxz}
    \left\{
    \begin{aligned}
&\abs{g}_{X_{r}}^2\stackrel{\rm def}{=}\sum_{\alpha\in\mathbb Z_+^3}M_{r,\alpha}^2 \norm{D^{\alpha} g}_{\mathcal{H}^1}^2,\\
&\abs{g}_{Z_{r}}^2\stackrel{\rm def}{=}\sum_{\alpha\in\mathbb Z_+^3}M_{r,\alpha}^2 \big(\norm{\comi zD^{\alpha} \partial_z^2g}_{H_x^1L_z^2}^2+\norm{\comi zD^{\alpha}\partial_xg}_{H_x^1L_z^2}^2\big),
    \end{aligned}
    \right.
\end{equation}
where, here and below,  for any multi-index $\alpha=(\alpha_1,\alpha_2,\alpha_3)\in\mathbb Z_+^3,$
\begin{equation}\label{def:Gamma}
      D^\alpha\stackrel{\rm def}{=}t^{\alpha_1+\alpha_2+\frac{\alpha_3}{2}}\partial_t^{\alpha_1}\partial_x^{\alpha_2}\partial_z^{\alpha_3}, 
\end{equation}
and 
\begin{equation}\label{def:Hthetam}
   M_{r,\alpha}\stackrel{\rm def}{=}\frac{r^{\abs\alpha}(\abs\alpha+1)^4}{\abs\alpha!}. 
\end{equation}
\end{definition}

 With the norms given above, we now state the main result concerning the analytic smoothing effect   as follows.

\begin{proposition}\label{prop:y}
Suppose the initial datum $u_0$ satisfies the assumptions in  Theorem \ref{thm:wellposedness} and let 
  $u\in L^\infty([0,+\infty[;\mathcal{H}^1)$ be the solution to system \eqref{eq:2D}, constructed in Theorem \ref{thm:wellposedness}  and satisfying estimate \eqref{ret:pri}. Then there exists a small constant $0<r<1$ such that, shrinking the number $\varepsilon_0$ in Theorem \ref{thm:wellposedness} if necessary,
 \begin{equation}\label{est:y}
 \forall\ t\ge0,\quad   \abs{u(t)}_{X_{r}}^2+\int^t_0\abs{u(s)}_{Z_{r}}^2ds\leq \varepsilon_0^2,
 \end{equation}
where the norms $\abs{\cdot}_{X_{r}}$ and $\abs{\cdot}_{Z_{r}}$ are defined as in \eqref{dxz}.
\end{proposition}

Before  proving Proposition \ref{prop:y}, we first list several estimates that will be used frequently.
By the definition of the $\mathcal H^1$-norm in \eqref{def:weightnorm}, 
\begin{equation*}
	\norm{D^{\alpha} \partial_zg}_{\mathcal{H}^1}^2=\norm{D^{\alpha} \partial_z g}_{H_x^1L_z^2}^2+\norm{\comi z D^{\alpha} \partial_z^2g}_{H_x^1L_z^2}^2,
\end{equation*}
which, together with Hardy's inequality, implies 
\begin{equation}\label{xz}
  \abs{\partial_zg}_{X_{r}}^2\leq C \abs{g}_{Z_{r}}^2. 
\end{equation}
From the definition of $M_{r,\alpha}$ in \eqref{def:Hthetam}, it follows that  for any multi-indices $\alpha,\beta\in\mathbb Z_+^3,$ 
\begin{equation}\label{sdf}
  M_{r,\alpha}= M_{r,\beta} \ \textrm{ if }\  \abs\beta=\abs\alpha,\quad \abs\alpha M_{r,\alpha}\leq  Cr M_{r,\beta} \ \textrm{ if }\  \abs{\beta}=\abs{\alpha}-1.
\end{equation}
Recall  that $\tilde\alpha=\alpha-(1,0,0)$ and $\alpha_*=\alpha-(0,0,1)$ for $\alpha=(\alpha_1,\alpha_2,\alpha_3)\in\mathbb Z_+^3$.    Then 
\begin{equation}\label{facttwo}
 D^\alpha=tD^{\tilde\alpha}\partial_t   \  \textrm{ and }\  D^\alpha=t^{\frac12}D^{\alpha_*}\partial_z.
\end{equation}
 We will  use the  following version of Young's inequality for discrete convolution:
\begin{equation}\label{young}
\bigg[\sum^{+\infty}_{m=0}\Big{(}\sum^m_{j=0}p_jq_{m-j}\Big{)}^2\bigg]^{\frac12}\leq \Big{(}\sum^{+\infty}_{m=0}q_m^2\Big{)}^{\frac12}\sum^{+\infty}_{j=0}p_j,
\end{equation}
where $\{p_j\}_{j\ge 0}$ and $\{q_j\}_{j\ge 0}$ are   sequences of nonnegative real numbers.
 
We now begin to prove Proposition \ref{prop:y}. For given multi-index $\alpha=(\alpha_1,\alpha_2,\alpha_3)\in\mathbb Z_+^3$, we apply $D^{\alpha}$ to the velocity equation in system \eqref{eq:2D} and then take the $\mathcal{H}^1$-product with $D^{\alpha} u$ to derive that
\begin{multline}\label{kp}
\frac{1}{2}\frac{d}{dt}\norm{D^{\alpha} u}_{\mathcal{H}^1}^2-\big(D^{\alpha}\partial_z^2u+D^{\alpha}\partial_xf,\ D^{\alpha} u\big)_{\mathcal{H}^1}\\
=-\inner{D^{\alpha}(u\partial_xu+w\partial_zu),\ D^{\alpha} u}_{\mathcal{H}^1}+\frac{2\alpha_1+2\alpha_2+\alpha_3}{2t}\norm{D^{\alpha} u}_{\mathcal{H}^1}^2.        
\end{multline}
Using integration by parts yields
\begin{align*}
 -\big(D^{\alpha}\partial_z^2u,\ D^{\alpha} u\big)_{H_x^1L_z^2}=\norm{D^{\alpha}\partial_zu}_{H_x^1L_z^2}^2+  \inner{D^{\alpha}\partial_zu|_{z=0},\ D^{\alpha} u|_{z=0}}_{H_x^1},  
\end{align*}
and
\begin{align*}
&-\big(\comi zD^{\alpha}\partial_z^3u,\ \comi zD^{\alpha}\partial_z u\big)_{H_x^1L_z^2}\\
&= \norm{\comi zD^{\alpha}\partial_z^2u}_{H_x^1L_z^2}^2+2\big(zD^{\alpha}\partial_z^2u,\ D^{\alpha}\partial_zu\big)_{H_x^1L_z^2}+\big(D^{\alpha}\partial_z^2u|_{z=0},\ D^{\alpha} \partial_zu|_{z=0}\big)_{H_x^1}\\
&= \norm{\comi zD^{\alpha}\partial_z^2u}_{H_x^1L_z^2}^2-\norm{D^{\alpha}\partial_zu}_{H_x^1L_z^2}^2+\big(D^{\alpha}\partial_z^2u|_{z=0},\ D^{\alpha} \partial_zu|_{z=0}\big)_{H_x^1}.
\end{align*}
Recalling the definition  of $\norm{\cdot}_{\mathcal{H}^1}$ in \eqref{def:weightnorm}  and combining these identities, we obtain
\begin{equation}\label{modelone}
\begin{aligned}
   -\big(D^{\alpha}\partial_z^2u,\ D^{\alpha} u\big)_{\mathcal{H}^1}=& \norm{\comi zD^{\alpha}\partial_z^2u}_{H_x^1L_z^2}^2+\sum_{k=0}^1\big(D^{\alpha}\partial_z^{k+1}u|_{z=0},\ D^{\alpha} \partial_z^ku|_{z=0}\big)_{H_x^1}.
\end{aligned}
\end{equation}
Similarly, we derive from  $\partial_xu+\partial_z^2f=0$ that 
\begin{equation*}
   -\big(D^{\alpha}\partial_xf,\ D^{\alpha} u\big)_{\mathcal{H}^1}=-\big(D^{\alpha} f,\ D^{\alpha} \partial_z^2f\big)_{H_x^1L_z^2}-\big(\comi zD^{\alpha} \partial_zf,\ \comi zD^{\alpha} \partial_z^3f\big)_{H_x^1L_z^2},
\end{equation*}
which enables us to  repeat the proof of \eqref{modelone} to conclude that
\begin{align*}
   -\inner{D^{\alpha}\partial_xf,\ D^{\alpha} u}_{\mathcal{H}^1}=& \norm{\comi zD^{\alpha}\partial_z^2f}_{H_x^1L_z^2}^2
   +\sum_{k=0}^1\big(D^{\alpha}\partial_z^{k+1}f|_{z=0},\ D^{\alpha} \partial_z^kf|_{z=0}\big)_{H_x^1}\\
   =& \norm{\comi zD^{\alpha}\partial_xu}_{H_x^1L_z^2}^2
   +\sum_{k=0}^1\big(D^{\alpha}\partial_z^{k+1}f|_{z=0},\ D^{\alpha} \partial_z^kf|_{z=0}\big)_{H_x^1}.
\end{align*}
Substituting  this and  \eqref{modelone} into \eqref{kp},    then multiplying by $M_{r,\alpha}^2$ and summing over $\alpha\in\mathbb Z_+^3$, we obtain  
\begin{multline*} 
\frac{1}{2}\frac{d}{dt}\sum_{ \alpha\in\mathbb Z_+^3 }M_{r,\alpha}^2 \norm{D^{\alpha} u}_{\mathcal{H}^1}^2+\sum_{ \alpha\in\mathbb Z_+^3 }M_{r,\alpha}^2 \norm{\comi zD^{\alpha}\partial_z^2u}_{H_x^1L_z^2}^2\\
+\sum_{ \alpha\in\mathbb Z_+^3 } M_{r,\alpha}^2 \norm{\comi zD^{\alpha}\partial_xu}_{H_x^1L_z^2}^2\leq \sum_{j=1}^5S_j,
\end{multline*}
that is,  recalling the definitions of  $\abs{\cdot}_{X_{r}}$ and  $\abs{\cdot}_{Z_{r}}$ in \eqref{dxz},
\begin{equation}\label{est:yy}
	\frac{1}{2}\frac{d}{dt}\abs{u}_{X_r}^2+\abs{u}_{Z_r}^2\leq \sum_{j=1}^5S_j, 
\end{equation}
where
\begin{equation}\label{J1-5}
    \left\{
    \begin{aligned}
&S_1=-\sum_{\substack{\alpha\in\mathbb Z_+^3\\ \alpha_3=0}}M_{r,\alpha}^2\inner{D^{\alpha}(u\partial_xu+w\partial_zu),\ D^{\alpha} u}_{\mathcal{H}^1},\\
&S_2=-\sum_{\substack{\alpha\in\mathbb Z_+^3\\ \alpha_3\ge 1}}M_{r,\alpha}^2\inner{D^{\alpha}(u\partial_xu+w\partial_zu),\ D^{\alpha} u}_{\mathcal{H}^1},\\
&S_3=-\sum_{0\leq k\leq 1}\sum_{\alpha\in\mathbb Z_+^3}M_{r,\alpha}^2\big(D^{\alpha}\partial_z^{k+1}u|_{z=0},\ D^{\alpha} \partial_z^ku|_{z=0}\big)_{H_x^1},\\
&S_4=-\sum_{0\leq k\leq 1}\sum_{\alpha\in\mathbb Z_+^3}M_{r,\alpha}^2\big(D^{\alpha}\partial_z^{k+1}f|_{z=0},\ D^{\alpha} \partial_z^kf|_{z=0}\big)_{H_x^1},\\
&S_5=\sum_{\alpha\in\mathbb Z_+^3}\frac{2\alpha_1+2\alpha_2+\alpha_3}{2t}M_{r,\alpha}^2\norm{D^{\alpha} u}_{\mathcal{H}^1}^2.
    \end{aligned}
    \right.
\end{equation}

The rest of this subsection is devoted to estimating the terms $S_j$ for $1\leq j\leq 5$.  The proofs of these estimates are presented in the following five lemmas.

\begin{lemma}
	[Estimate on $S_1$] \label{j1++}
	Let $S_1$ be given in \eqref{J1-5}. It holds that
	\begin{equation*}
		S_1\leq C\abs{u}_{X_r}\abs{u}_{Z_r}^2,
	\end{equation*}
where the norms $\abs{\cdot}_{{X}_r}$ and $\abs{\cdot}_{{Z}_r}$ are defined as in   \eqref{dxz}.     
\end{lemma}

\begin{proof}
 For fixed multi-index $\alpha=(\alpha_1,\alpha_2,\alpha_3)\in\mathbb Z_+^3$ with $\alpha_3=0$, Leibniz's formula gives
\begin{multline*}
	-\inner{D^{\alpha}(u\partial_xu+w\partial_zu),\ D^{\alpha}u}_{\mathcal{H}^1}\\
 =-\sum_{\beta\leq\alpha}\binom{\alpha}{\beta}\big((D^{\beta}u)D^{\alpha-\beta}\partial_xu+(D^{\alpha-\beta}w)D^{\beta}\partial_zu,\ D^{\alpha}u\big)_{\mathcal{H}^1}.
\end{multline*}
By repeating an argument analogous to that after \eqref{eses}, we obtain that for $\alpha=(\alpha_1,\alpha_2,\alpha_3)\in\mathbb Z_+^3$ with $\alpha_3=0$,
\begin{multline*}
	\big|\big((D^{\beta}u)D^{\alpha-\beta}\partial_xu+(D^{\alpha-\beta}w)D^{\beta}\partial_zu,\ D^{\alpha}u\big)_{H_x^1L_z^2}	\big|\\
 \leq  C\norm{\comi zD^{\beta}\partial_z^2u}_{H_x^1L_z^2}\norm{\comi z D^{\alpha-\beta}\partial_xu}_{H_x^1L_z^2}\norm{D^{\alpha} u}_{\mathcal{H}^1},
\end{multline*}
and 
\begin{align*}
	&\big|\big(\comi z \partial_z\big [(D^{\beta}u)D^{\alpha-\beta}\partial_xu+(D^{\alpha-\beta}w)D^{\beta}\partial_zu\big],\ \comi z\partial_z D^{\alpha}u\big)_{H_x^1L_z^2}	\big|\\
	&\leq C\Big(\norm{ D^{\beta}u}_{H_x^1H_z^1} +\norm{ \comi z \partial_z D^{\beta}u}_{H_x^1L_z^2}\Big) \norm{\comi z  D^{\alpha-\beta}\partial_xu}_{H_x^1L_z^2}\norm{\comi z\partial_z^2 D^{\alpha} u}_{H_x^1L_z^2}	  \\
	 & \leq C\norm{D^{\beta}u}_{\mathcal H^1}\norm{\comi z D^{\alpha-\beta}\partial_xu}_{H_x^1L_z^2}\norm{\comi z\partial_z^2 D^{\alpha} u}_{H_x^1L_z^2}.
	 \end{align*}
Hence, recalling   definition \eqref{def:weightnorm} of $\norm{\cdot}_{\mathcal H^1},$ we combine the above estimates to obtain 
 \begin{equation}\label{decom:S1}
     \begin{aligned}
S_1\leq& C \sum_{\alpha\in\mathbb Z_+^3}\sum_{\beta\leq\alpha}\binom{\alpha}{\beta}M_{r,\alpha}^2\norm{\comi zD^{\beta}\partial_z^2u}_{H_x^1L_z^2}\norm{\comi zD^{\alpha-\beta}\partial_xu}_{H_x^1L_z^2}\norm{D^{\alpha}u}_{\mathcal{H}^1}\\
&+C\sum_{\alpha\in\mathbb Z_+^3}\sum_{\beta\leq\alpha}\binom{\alpha}{\beta}M_{r,\alpha}^2\norm{D^{\beta}u}_{\mathcal{H}^1}\norm{\comi zD^{\alpha-\beta}\partial_xu}_{H_x^1L_z^2}\norm{\comi zD^{\alpha}\partial_z^2u}_{H_x^1L_z^2}\\
\stackrel{\rm def}{=}&S_{1,1}+S_{1,2}.
     \end{aligned}
 \end{equation}
 Observe for any multi-indices $\alpha,\beta\in\mathbb Z_+^3$ with $\beta\leq\alpha$,
 \begin{equation}
 	\label{factor}
 	\binom{\alpha}{\beta}\leq \binom{\abs \alpha}{\abs\beta}.
 \end{equation}
 Then  we have
 \begin{equation*}
 	\begin{aligned}
&\binom{\alpha}{\beta}\frac{M_{r,\alpha}}{M_{r,\beta}M_{r,\alpha-\beta}}\\
&\leq \frac{\abs\alpha!}{\abs\beta!(\abs\alpha-\abs\beta)!}\frac{r^{\abs\alpha}(\abs\alpha+1)^4}{\abs\alpha!}\frac{\abs\beta!}{r^{\abs\beta}(\abs\beta+1)^4}\frac{(\abs\alpha-\abs\beta)!}{r^{\abs\alpha-\abs\beta}(\abs\alpha-\abs\beta+1)^4}\\
&\leq \frac{(\abs\alpha+1)^4}{(\abs\beta+1)^4(\abs\alpha-\abs\beta+1)^4}\leq \frac{C}{(\abs\beta+1)^4}+\frac{C}{(\abs\alpha-\abs\beta+1)^4}.
    \end{aligned}
 \end{equation*}
Combining this with the identity
 \begin{equation*}
  \binom{\alpha}{\beta} M_{r,\alpha} =	\binom{\alpha}{\beta}\frac{M_{r,\alpha}}{M_{r,\beta}M_{r,\alpha-\beta}}   M_{r,\beta} M_{r,\alpha-\beta},
 \end{equation*}
 we deduce that
\begin{equation}\label{S111}
    \begin{aligned}
&S_{1,1} \leq C\bigg[\sum_{\alpha\in\mathbb Z_+^3}\bigg(\sum_{\beta\leq\alpha}\binom{\alpha}{\beta}M_{r,\alpha}\norm{\comi zD^{\beta}\partial_z^2u}_{H_x^1L_z^2}\norm{\comi zD^{\alpha-\beta}\partial_xu}_{H_x^1L_z^2}\bigg)^2\bigg]^{\frac12}\abs{u}_{X_r} \\
&\leq C\bigg[\sum_{\alpha\in\mathbb Z_+^3} \bigg(\sum_{\beta\leq\alpha}\frac{M_{r,\beta}\norm{\comi zD^{\beta}\partial_z^2u}_{H_x^1L_z^2}}{(\abs\beta+1)^4}M_{r,\alpha-\beta}\norm{\comi zD^{\alpha-\beta}\partial_xu}_{H_x^1L_z^2}\bigg)^2\bigg]^\frac{1}{2}\abs{u}_{X_r}\\
&\ \ +  C\bigg[\sum_{\alpha\in\mathbb Z_+^3} \bigg(\sum_{\beta\leq\alpha}M_{r,\beta}\norm{\comi zD^{\beta}\partial_z^2u}_{H_x^1L_z^2}\frac{M_{r,\alpha-\beta}\norm{\comi zD^{\alpha-\beta}\partial_xu}_{H_x^1L_z^2}}{(\abs\alpha-\abs\beta+1)^4}\bigg)^2\bigg]^\frac{1}{2}\abs{u}_{X_r},
    \end{aligned}
\end{equation}
where the first inequality uses the Cauchy inequality and the definition  of $\abs{\cdot}_{X_r}$ in \eqref{dxz}. 
Moreover, by Young's inequality \eqref{young} for discrete convolution and the definition of $\abs{\cdot}_{Z_r}$ in \eqref{dxz}, one has
 \begin{equation}\label{applyyoung}
     \begin{aligned}
&\bigg[\sum_{\alpha\in\mathbb Z_+^3} \bigg(\sum_{\beta\leq\alpha}\frac{M_{r,\beta}\norm{\comi zD^{\beta}\partial_z^2u}_{H_x^1L_z^2}}{(\abs\beta+1)^4}M_{r,\alpha-\beta}\norm{\comi zD^{\alpha-\beta}\partial_xu}_{H_x^1L_z^2}\bigg)^2\bigg]^\frac{1}{2}\\
&\leq C\bigg(\sum_{\alpha\in\mathbb Z_+^3}\frac{M_{r,\alpha}\norm{\comi zD^{\alpha}\partial_z^2u}_{H_x^1L_z^2}}{(\abs\alpha+1)^4}\bigg)\bigg(\sum_{\alpha\in\mathbb Z_+^3}M_{r,\alpha}^2\norm{\comi zD^{\alpha}\partial_xu}_{H_x^1L_z^2}^2\bigg)^\frac{1}{2}\\
&\leq C\bigg(\sum_{\alpha\in\mathbb Z_+^3}M_{r,\alpha}^2\norm{\comi zD^{\alpha}\partial_z^2u}_{H_x^1L_z^2}^2\bigg)^\frac{1}{2}\abs{u}_{Z_r}\leq C\abs{u}_{Z_r}^2.
     \end{aligned}
 \end{equation}
 Therefore, the first term on the right-hand side of \eqref{S111} is bounded above by $C\abs{u}_{X_r}\abs{u}_{Z_r}^2$, and the last term admits the same bound. This yields the estimate for $S_{1,1}$ in \eqref{decom:S1}:
\begin{equation*}%\label{est:S11}
S_{1,1}\leq C\abs{u}_{X_r}\abs{u}_{Z_r}^2.
\end{equation*}
A similar argument applied to $S_{1,2}$ in \eqref{decom:S1} gives
\begin{equation*}
    S_{1,2}\leq C\abs{u}_{X_r}\abs{u}_{Z_r}^2.
\end{equation*}
Substituting the two estimates into \eqref{decom:S1} yields the desired estimate in Lemma \ref{j1++}. This completes the proof.
\end{proof}

\begin{lemma}[Estimate on $S_2$]\label{lem:J1}
    Let $S_2$ be given in \eqref{J1-5}, namely,
\begin{equation*}
    S_2=-\sum_{\substack{\alpha\in\mathbb Z_+^3\\ \alpha_3\ge 1}}M_{r,\alpha}^2\inner{D^{\alpha}(u\partial_xu+w\partial_zu),\ D^{\alpha} u}_{\mathcal{H}^1}.
\end{equation*}  
It holds that
    \begin{equation}\label{est:J1}
S_2\leq C\abs{u}_{{X}_r}\abs{u}_{{Z}_r}^2,      
    \end{equation}
where the norms $\abs{\cdot}_{{X}_r}$ and $\abs{\cdot}_{{Z}_r}$ are defined as in   \eqref{dxz}.    
\end{lemma}

\begin{proof}
Fix $\alpha=(\alpha_1,\alpha_2,\alpha_3)\in\mathbb Z_+^3$ with $\alpha_3\ge 1$, and recall   $\alpha_*=\alpha-(0,0,1).$  Using \eqref{sdf} and \eqref{facttwo},
we estimate 
$S_2$	
  as follows:
\begin{align*}
    S_2\leq &\sum_{\substack{\alpha\in\mathbb Z_+^3\\ \alpha_3\ge 1}}M_{r,\alpha}^2\norm{D^{\alpha}(u\partial_xu+w\partial_zu)}_{\mathcal{H}^1}\norm{D^{\alpha} u}_{\mathcal{H}^1}\\
    \leq&C\sum_{\substack{\alpha\in\mathbb Z_+^3\\ \alpha_3\ge 1}}\big(t^{\frac{1}{2}}r\abs{\alpha}^{-1}M_{r,\alpha}\norm{D^{\alpha}(u\partial_xu+w\partial_zu)}_{\mathcal{H}^1}\big)\big(M_{r,\alpha_*}\norm{D^{\alpha_*}\partial_z u}_{\mathcal{H}^1}\big)\\
    \leq &C\bigg[\sum_{\substack{\alpha\in\mathbb Z_+^3\\ \alpha_3\ge 1}}\Big(t^{\frac12}r\abs{\alpha}^{-1}M_{r,\alpha}\norm{D^{\alpha}(u\partial_xu+w\partial_zu)}_{\mathcal{H}^1}\Big)^2\bigg]^\frac{1}{2}\abs{\partial_zu}_{X_{r}}.
\end{align*}
For the   last factor, using estimate \eqref{xz} yields 
\begin{equation}\label{df}
	\abs{\partial_zu}_{X_{r}}\leq C\abs{u}_{Z_{r}}.
\end{equation}
Thus, assertion \eqref{est:J1} follows once we establish the inequality
\begin{equation}\label{claimone}
   \bigg[\sum_{\substack{\alpha\in\mathbb Z_+^3\\ \alpha_3\ge 1}}\Big(t^{\frac12}r\abs{\alpha}^{-1}M_{r,\alpha}\norm{D^{\alpha}(u\partial_xu+w\partial_zu)}_{\mathcal{H}^1}\Big)^2\bigg]^\frac{1}{2}\leq C\abs{u}_{X_{r}}\abs{u}_{Z_{r}}.
\end{equation}
 We now proceed to prove \eqref{claimone}  through two steps. 

\textit{Step 1}. In this step we will prove that
\begin{equation}\label{uxu}
   \bigg[\sum_{\substack{\alpha\in\mathbb Z_+^3\\ \alpha_3\ge 1}}\Big(t^{\frac12}r\abs{\alpha}^{-1}M_{r,\alpha}\norm{D^{\alpha}(u\partial_xu)}_{\mathcal{H}^1}\Big)^2\bigg]^\frac{1}{2}\leq C\abs{u}_{X_{r}}\abs{u}_{Z_{r}}.
\end{equation}
Since  $\mathcal{H}^1(\mathbb R_+^2)$  is an algebra under pointwise multiplication,  using  Leibniz's formula and the fact $0<r<1$ yields 
\begin{multline}\label{ef}
 \bigg[\sum_{\substack{\alpha\in\mathbb Z_+^3\\ \alpha_3\ge 1}}tr^2\abs{\alpha}^{-2}M_{r,\alpha}^2\norm{D^{\alpha}(u\partial_xu)}_{\mathcal{H}^1}^2\bigg]^\frac{1}{2}	\\
\leq \bigg[\sum_{\substack{\alpha\in\mathbb Z_+^3\\ \alpha_3\ge 1}}\bigg(\sum_{\beta\leq\alpha}\binom{\alpha}{\beta}t^{\frac12}\abs{\alpha}^{-1}M_{r,\alpha}\norm{D^\beta u}_{\mathcal{H}^1}\norm{D^{\alpha-\beta}\partial_xu}_{\mathcal{H}^1}\bigg)^2 \bigg]^\frac{1}{2}.
\end{multline}
For any given multi-index $\beta=(\beta_1,\beta_2,\beta_3)\in\mathbb Z_+^3$ with $\beta\leq \alpha$, if $\beta_3=0$,  the condition $\alpha_3\geq 1$ enables us to write
\begin{equation*}
	 \norm{D^{\alpha-\beta}\partial_xu}_{\mathcal{H}^1}=t^{-\frac12} \norm{D^{\alpha-\beta+(0,1,-1)}\partial_zu}_{\mathcal{H}^1},
\end{equation*}
where we used the definition \eqref{def:Gamma} of $D^{\alpha}$.  If $\beta_3\geq 1$, we have, recalling $\beta_*=\beta-(0,0,1)$,
\begin{equation*}
	\norm{D^\beta u}_{\mathcal{H}^1}=t^{\frac12}\norm{D^{\beta_*} \partial_z u}_{\mathcal{H}^1}\  \textrm{ and }\  \norm{D^{\alpha-\beta}\partial_xu}_{\mathcal{H}^1}= t^{-1} \norm{D^{\alpha-\beta+(0,1,0)}u}_{\mathcal{H}^1}.
\end{equation*}
Combining these estimates, we obtain 
\begin{equation}\label{tr}
 \begin{aligned}
&\sum_{\beta\leq\alpha}\binom{\alpha}{\beta}t^{\frac12}\abs{\alpha}^{-1}M_{r,\alpha}\norm{D^\beta u}_{\mathcal{H}^1}\norm{D^{\alpha-\beta}\partial_xu}_{\mathcal{H}^1}\\
&= \sum_{\substack{\beta\leq\alpha\\ \beta_3=0}}
 \binom{\alpha}{\beta} \abs{\alpha}^{-1}M_{r,\alpha}\norm{D^\beta u}_{\mathcal{H}^1}\norm{D^{\alpha-\beta+(0,1,-1)}\partial_zu}_{\mathcal{H}^1}\\
&\quad+ \sum_{\substack{\beta\leq\alpha\\ \beta_3\geq 1}}\binom{\alpha}{\beta} \abs{\alpha}^{-1} M_{r,\alpha}  \norm{D^{\beta_*}\partial_z u}_{\mathcal{H}^1} \norm{D^{\alpha-\beta+(0,1,0)}u}_{\mathcal{H}^1}.
 \end{aligned}\
 \end{equation}   
On the other hand, a direct computation (see Appendix \ref{sec:ineq} for details) shows that
 \begin{equation}\label{ineq3}
 	\begin{aligned}
 	\binom{\alpha}{\beta}  \frac{\abs{\alpha}^{-1}M_{r,\alpha} }{M_{r, \beta}M_{r, \alpha-\beta+(0,1,-1)}}\leq \frac{C}{(\abs{\beta}+1)^4}+\frac{C}{(\abs{\alpha}-\abs{\beta}+1)^4}\quad \mathrm{if}\quad \beta_3=0,
 	\end{aligned}
 \end{equation} 
 and 
 \begin{equation}\label{ineq4}
 	\begin{aligned}
 	\binom{\alpha}{\beta}  \frac{\abs{\alpha}^{-1}M_{r,\alpha} }{M_{r,\beta_*}M_{r,\alpha-\beta+(0,1,0)}}\leq \frac{C}{\abs{\beta}^4}+\frac{C}{(\abs{\alpha}-\abs{\beta}+2)^4} \quad \mathrm{if}\quad \beta_3\ge 1.
 	\end{aligned}
 \end{equation} 
Combining these inequalities  with \eqref{tr},  we repeat the argument in \eqref{S111} and \eqref{applyyoung} to conclude that 
\begin{multline}\label{ew}
\bigg[\sum_{\substack{\alpha\in\mathbb Z_+^3\\ \alpha_3\ge 1}}\bigg(\sum_{\beta\leq\alpha}\binom{\alpha}{\beta}t^{\frac12}\abs{\alpha}^{-1}M_{r,\alpha}\norm{D^\beta u}_{\mathcal{H}^1}\norm{D^{\alpha-\beta}\partial_xu}_{\mathcal{H}^1}\bigg)^2 \bigg]^\frac{1}{2}\\
\leq C\abs{u}_{X_{r}}\abs{\partial_zu}_{X_{r}} \leq C\abs{u}_{X_{r}}\abs{u}_{Z_{r}},   
\end{multline}
the last inequality following from \eqref{df}. Combining this with \eqref{ef}
gives the desired estimate \eqref{uxu}.

\textit{Step 2}. This step is devoted to proving the estimate  
\begin{equation}\label{vyu}
    \bigg[\sum_{\substack{\alpha\in\mathbb Z_+^3\\ \alpha_3\ge 1}}\Big(t^{\frac12}r\abs{\alpha}^{-1}M_{r,\alpha}\norm{D^{\alpha}(w\partial_zu)}_{\mathcal{H}^1}\Big)^2\bigg]^\frac{1}{2}\leq C\abs{u}_{X_{r}}\abs{u}_{Z_{r}}.
\end{equation}
Let $\beta=(\beta_1,\beta_2,\beta_3)\in\mathbb Z_+^3$ be any given multi-index satisfying $\beta\leq\alpha$. 
If $\beta_3=0$, applying  the Sobolev inequality yields
\begin{multline*}
 \norm{(D^\beta w)D^{\alpha-\beta}\partial_zu}_{\mathcal{H}^1}\leq C \norm{\comi zD^\beta\partial_x u}_{H_x^1L_z^2}\norm{D^{\alpha-\beta}\partial_zu}_{\mathcal{H}^1}\\
 \leq  C t^{-\frac12}\norm{\comi zD^\beta\partial_x u}_{H_x^1L_z^2}\norm{D^{\alpha-\beta+(0,0,1)} u}_{\mathcal{H}^1}.
    \end{multline*}
If $\beta_3\ge 1$, it follows from the fact $\partial_xu+\partial_zw=0$ that  
\begin{equation*}
D^\beta w=-t^{-\frac{1}{2}}D^{\beta+(0,1,-1)}u.
\end{equation*} 
Thus, we combine the estimates above to obtain
 \begin{align*}
&t^\frac{1}{2}r\abs{\alpha}^{-1}M_{r,\alpha}\norm{D^{\alpha}(w\partial_zu)}_{\mathcal{H}^1}\\
&\leq C \sum_{\substack{\beta\leq\alpha\\ \beta_3=0}}\binom{\alpha}{\beta} r \abs{\alpha}^{-1} M_{r,\alpha} \norm{\comi zD^\beta \partial_xu}_{H_x^1L_z^2}  \norm{D^{\alpha-\beta+(0,0,1)}u}_{\mathcal{H}^1} \\
&\quad+C \sum_{\substack{\beta\leq\alpha\\
\beta_3\ge1}}\binom{\alpha}{\beta} \abs{\alpha}^{-1} M_{r,\alpha} \norm{D^{\beta+(0,1,-1)} u}_{\mathcal{H}^1} \norm{D^{\alpha-\beta}\partial_zu}_{\mathcal{H}^1}.
 \end{align*} 
Moreover, a direct computation (see Appendix \ref{sec:ineq} for details) shows that
 \begin{equation}\label{ineq5}
 	\binom{\alpha}{\beta}\frac{r \abs{\alpha}^{-1}M_{r,\alpha}}{M_{r,\beta}M_{r,\alpha-\beta+(0,0,1)}} \leq  \frac{C}{(\abs{\beta}+1)^4}+\frac{C}{(\abs{\alpha}-\abs{\beta}+2)^4}\quad \mathrm{if}\quad \beta_3=0,
 \end{equation}
 and
 \begin{equation}
 	\label{ineq6}
 	\binom{\alpha}{\beta}\frac{\abs{\alpha}^{-1}M_{r,\alpha}}{M_{r, \beta+(0,1,-1)}M_{r,\alpha-\beta}} \leq \frac{C}{(\abs{\beta}+1)^4}+\frac{C}{(\abs{\alpha}-\abs{\beta}+1)^4} \quad \mathrm{if}\quad \beta_3\ge 1.
 \end{equation}
Hence, similar to \eqref{ew}, repeating the argument  in \eqref{S111} and \eqref{applyyoung}, we obtain 
\begin{multline*}
   \bigg[\sum_{\substack{\alpha\in\mathbb Z_+^3\\ \alpha_3\ge 1}}\Big(t^{\frac12}r\abs{\alpha}^{-1}M_{r,\alpha}\norm{D^{\alpha}(w\partial_zu)}_{\mathcal{H}^1}\Big)^2\bigg]^\frac{1}{2}\\
    \leq C\abs{u}_{X_{r}}\abs{u}_{Z_{r}}+C\abs{u}_{X_{r}}\abs{\partial_zu}_{X_{r}}\leq C\abs{u}_{X_{r}}\abs{u}_{Z_{r}}.
\end{multline*}
  Then estimate \eqref{vyu} follows. Combining \eqref{uxu} and \eqref{vyu} yields assertion \eqref{claimone} and thus completes the proof of Lemma \ref{lem:J1}.
\end{proof}

\begin{lemma}[Estimate on $S_3$]\label{lem:J3}
    Let $S_3$ be given in \eqref{J1-5}, namely,  
\begin{equation*}
 S_3=-\sum_{0\leq k\leq 1}\sum_{\alpha\in\mathbb Z_+^3}M_{r,\alpha}^2\big(D^{\alpha}\partial_z^{k+1}u|_{z=0},\ D^{\alpha} \partial_z^ku|_{z=0}\big)_{H_x^1}.   
\end{equation*}    
    It holds that
    \begin{equation}\label{est:J2}
S_3\leq Cr^\frac{1}{2}\abs{u}_{{Z}_r}^2+C\abs{u}_{{X}_r}^\frac{1}{2}\abs{u}_{{Z}_r}^2,      
    \end{equation}
where the norms $\abs{\cdot}_{{X}_r}$ and $\abs{\cdot}_{{Z}_r}$ are defined as in \eqref{dxz}.    
\end{lemma}

\begin{proof}
Recall that $\alpha_* = \alpha - (0,0,1)$ for $\alpha \in \mathbb {Z}_+^3$ with $\alpha_3 \ge 1$.
Observing $ \partial_z^2 u|_{z=0} = 0$, we use the Sobolev inequality to obtain 
\begin{align*}
&-\sum_{ \alpha\in\mathbb Z_+^3 }M_{r,\alpha}^2(D^{\alpha}\partial_z^{2}u|_{z=0},  D^{\alpha} \partial_zu|_{z=0})_{H_x^1}=-\sum_{\substack{\alpha\in\mathbb Z_+^3\\ \alpha_3\ge 2}}M_{r,\alpha}^2(D^{\alpha}\partial_z^{2}u|_{z=0},  D^{\alpha} \partial_zu|_{z=0})_{H_x^1}\\
&\leq C\sum_{\substack{\alpha\in\mathbb Z_+^3\\ \alpha_3\ge 2}}M_{r,\alpha}^2\norm{D^{\alpha}\partial_z^{3}u}_{H_x^1L_z^2}^\frac{1}{2}\norm{D^{\alpha}\partial_z^{2}u}_{H_x^1L_z^2}\norm{D^{\alpha}\partial_zu}_{H_x^1L_z^2}^\frac{1}{2}\\
    &\leq C \sum_{\substack{\alpha\in\mathbb Z_+^3\\ \alpha_3\ge 2}}\big(t^{\frac{1}{4}}r^\frac{1}{2}\abs{\alpha}^{-\frac{1}{2}}M_{r,\alpha}^\frac{1}{2}\norm{D^{\alpha}\partial_z^{3}u}_{H_x^1L_z^2}^\frac{1}{2}\big)\big(M_{r,\alpha}\norm{\comi zD^{\alpha}\partial_z^{2}u}_{H_x^1L_z^2}\big)\\
    &\qquad\times\big(M_{r,\alpha_*}^\frac{1}{2}\norm{\comi zD^{\alpha_*}\partial_z^{2}u}_{H_x^1L_z^2}^\frac{1}{2}\big)\\
 &\leq    C\bigg[\sum_{\substack{\alpha\in\mathbb Z_+^3\\ \alpha_3\ge 2}}\Big(t^{\frac12}r\abs{\alpha}^{-1}M_{r,\alpha}\norm{D^{\alpha}\partial_z^3u}_{H_x^1L_z^2}\Big)^2\bigg]^\frac{1}{4}\abs{u}_{{Z}_r}^\frac{3}{2},
\end{align*}
where the second inequality uses \eqref{sdf} and \eqref{facttwo},  and the last one follows from 
 the definition of $\abs{\cdot}_{{Z}_r}$. On the other hand,
using an analogous argument and the boundary condition  $u|_{z=0} = \partial_z^2 u|_{z=0} = 0$, we have
\begin{align*}
&-\sum_{\alpha\in\mathbb Z_+^3}M_{r,\alpha}^2\inner{D^{\alpha}\partial_zu|_{z=0},\ D^{\alpha} u|_{z=0}}_{H_x^1}=-\sum_{\substack{\alpha\in\mathbb Z_+^3\\ \alpha_3\ge 3}}M_{r,\alpha}^2\inner{D^{\alpha}\partial_zu|_{z=0},\ D^{\alpha} u|_{z=0}}_{H_x^1}\\
&\leq C\sum_{\substack{\alpha\in\mathbb Z_+^3\\ \alpha_3\ge 3}}M_{r,\alpha}^2\norm{D^{\alpha}\partial_z^2u}_{H_x^1L_z^2}^\frac{1}{2}\norm{D^{\alpha}\partial_zu}_{H_x^1L_z^2}\norm{D^{\alpha} u}_{H_x^1L_z^2}^\frac{1}{2}\\
&\leq C\sum_{\substack{\alpha\in\mathbb Z_+^3\\ \alpha_3\ge 3}} t^{\frac14}r^{\frac12}\abs\alpha^{-\frac12}M_{r,\alpha}^{\frac32}\norm{D^{\alpha}\partial_z^2u}_{H_x^1L_z^2}^\frac{1}{2}\norm{D^{\alpha}\partial_zu}_{H_x^1L_z^2}\big( M_{r,\alpha_*}\norm{D^{\alpha_*}\partial_z u}_{H_x^1L_z^2}\big)^\frac{1}{2}\\
&\leq C\bigg[\sum_{\substack{\alpha\in\mathbb Z_+^3\\ \alpha_3\ge 3}}\Big(t^\frac12r\abs{\alpha}^{-1}M_{r,\alpha}\norm{D^{\alpha}\partial_z^2u}_{H_x^1L_z^2}\Big)^2\bigg]^\frac{1}{4}\abs{u}_{{Z}_r}^\frac{3}{2},
\end{align*}
where the second inequality uses \eqref{sdf} and \eqref{facttwo} again,  and  the last line follows from  the Hardy's inequality which yields
\begin{equation}\label{hi}
\forall\ \gamma\in\mathbb Z_+^3,\quad 	\norm{D^{\gamma}\partial_zu}_{H_x^1L_z^2}\leq C \norm{zD^{\gamma}\partial_z^2u}_{H_x^1L_z^2} \leq C \norm{\comi zD^{\gamma}\partial_z^2u}_{H_x^1L_z^2}.
\end{equation}
Then combining the estimates above yields
\begin{equation}\label{decom:J2}
    S_3\leq C(R_1+R_2)^\frac{1}{2}\abs{u}_{{Z}_r}^\frac{3}{2},
\end{equation}
where
\begin{equation*}
\left\{
\begin{aligned}
&R_1=\bigg[\sum_{\substack{\alpha\in\mathbb Z_+^3\\ \alpha_3\ge 2}}\Big(t^{\frac12}r\abs{\alpha}^{-1}M_{r,\alpha}\norm{D^{\alpha}\partial_z^3u}_{H_x^1L_z^2}\Big)^2\bigg]^\frac{1}{2},\\
&R_2=\bigg[\sum_{\substack{\alpha\in\mathbb Z_+^3\\ \alpha_3\ge 3}}\Big(t^\frac12r\abs{\alpha}^{-1}M_{r,\alpha}\norm{D^{\alpha}\partial_z^2u}_{H_x^1L_z^2}\Big)^2\bigg]^\frac{1}{2}.
\end{aligned}
\right.
\end{equation*}
We now estimate $R_1$ and $R_2$ through the following two steps.

\textit{Step 1 (Estimate of $R_1$)}. Using  the identity  $\partial_z^2u=\partial_tu-\partial_xf+u\partial_xu+w\partial_zu$, we split $R_1$ as follows:
\begin{equation}\label{decom:R1}
    \begin{aligned}
R_1\leq& \bigg[\sum_{\substack{\alpha\in\mathbb Z_+^3\\ \alpha_3\ge 2}}\Big(t^{\frac12}r\abs{\alpha}^{-1}M_{r,\alpha}\norm{D^{\alpha}\partial_z\partial_tu-D^{\alpha}\partial_z\partial_xf}_{H_x^1L_z^2}\Big)^2\bigg]^\frac{1}{2} \\
&+\bigg[\sum_{\substack{\alpha\in\mathbb Z_+^3\\ \alpha_3\ge 2}}\Big(t^{\frac12}r\abs{\alpha}^{-1}M_{r,\alpha}\norm{D^{\alpha}\partial_z(u\partial_xu+w\partial_zu)}_{H_x^1L_z^2}\Big)^2\bigg]^\frac{1}{2}\\
\stackrel{\rm def}{=}& R_{1,1}+R_{1,2}.
    \end{aligned}
\end{equation}
By the definition  \eqref{def:Gamma} of $D^\alpha$,  for any $\alpha\in \mathbb Z_+^3$ with $\alpha_3\ge 2$,
\begin{equation*}
    D^{\alpha}\partial_z\partial_tu=t^{-\frac{1}{2}}D^{\alpha+(1,0,-1)}\partial_z^2u
\end{equation*}
and
\begin{equation*}
    D^{\alpha}\partial_z\partial_xf=t^{-\frac{1}{2}}D^{\alpha+(0,1,-1)}\partial_z^2f=-t^{-\frac{1}{2}}D^{\alpha+(0,1,-1)}\partial_xu,
\end{equation*}
where the last  equality uses the fact that  $\partial_xu+\partial_z^2f=0$. Moreover, by \eqref{sdf},
\begin{equation*}
    M_{r,\alpha}=M_{r,\alpha+(1,0,-1)}=M_{r,\alpha+(0,1,-1)}.
\end{equation*}
Therefore, recalling the definition of $\abs{\cdot}_{{Z}_r}$ and observing $\abs\alpha^{-1}\leq 1$ for $\alpha_3\geq 1$,
we  combine  the above identities  to  deduce that
\begin{multline*}
 R_{1,1}\leq Cr\bigg(\sum_{\substack{\alpha\in\mathbb Z_+^3\\ \alpha_3\ge 2}}M_{r,\alpha+(1,0,-1)}^2\norm{D^{\alpha+(1,0,-1)}\partial_z^2u}_{H_x^1L_z^2}^2\bigg)^\frac{1}{2}\\
 +Cr\bigg(\sum_{\substack{\alpha\in\mathbb Z_+^3\\ \alpha_3\ge 2}}M_{r,\alpha+(0,1,-1)}^2\norm{D^{\alpha+(0,1,-1)}\partial_xu}_{H_x^1 L_z^2}^2\bigg)^\frac{1}{2}\leq Cr\abs{u}_{{Z}_r}.
\end{multline*}
On the other hand, using estimate \eqref{claimone} as well as the definition \eqref{def:weightnorm} of $\norm{\cdot}_{\mathcal H^1}$ gives
\begin{equation*}
R_{1,2}\leq C\bigg[\sum_{\substack{\alpha\in\mathbb Z_+^3\\ \alpha_3\ge 2}}\Big(t^{\frac12}r\abs{\alpha}^{-1}M_{r,\alpha}\norm{D^{\alpha}(u\partial_xu+w\partial_zu)}_{\mathcal H^1}\Big)^2\bigg]^\frac{1}{2}
\leq C\abs{u}_{{X}_r}\abs{u}_{{Z}_r}.
\end{equation*}
 Substituting the two estimates above into \eqref{decom:R1}  we obtain
\begin{equation}\label{est:R1}
    R_1\leq Cr\abs{u}_{{Z}_r}+C\abs{u}_{{X}_r}\abs{u}_{{Z}_r}.
\end{equation}

\textit{Step 2 (Estimate of  $R_2$)}. The treatment of 
$R_2$ is analogous to the previous one, with slight modifications. For any $\alpha\in\mathbb Z_+^3$ with $\alpha_3\ge3$, we use \eqref{def:Gamma} and $\partial_z^2u=\partial_tu-\partial_xf+u\partial_xu+w\partial_zu$ to write
\begin{multline*}
  D^{\alpha}\partial_z^2u= D^{\alpha}\partial_tu-D^{\alpha}\partial_xf+D^{\alpha}(u\partial_xu+w\partial_zu)\\
  =t^{-\frac{1}{2}}D^{\alpha+(1,0,-1)}\partial_zu+t^{-\frac{1}{2}}D^{\alpha+(0,2,-3)}\partial_zu+D^{\alpha}(u\partial_xu+w\partial_zu),
\end{multline*}
where the last line uses the fact that $ \partial_z^2f=-\partial_xu.$ On the other hand,  it follows from  \eqref{sdf}  that, 
  for any $\alpha\in\mathbb Z_+^3$ with $\alpha_3\ge3$,  
\begin{equation*}
    M_{r,\alpha}\leq Cr M_{r,\alpha+(0,2,-3)}\leq CM_{r,\alpha+(0,2,-3)},
\end{equation*}
the last inequality using  $0<r<1$. Then following  the argument in the previous step and using  estimate \eqref{claimone}, we obtain 
\begin{equation*}%\label{est:R2}
    \begin{aligned}
  R_2&=\bigg[\sum_{\substack{\alpha\in\mathbb Z_+^3\\ \alpha_3\ge 3}}\Big(t^{\frac12}r\abs{\alpha}^{-1}M_{r,\alpha}\norm{D^{\alpha}\partial_z^2u}_{H_x^1L_z^2}\Big)^2\bigg]^\frac{1}{2}  \\
&\leq  Cr\bigg(\sum_{\substack{\alpha\in\mathbb Z_+^3\\ \alpha_3\ge 3}}M_{r,\alpha+(1,0,-1)}^2\norm{D^{\alpha+(1,0,-1)}\partial_zu}_{H_x^1L_z^2}^2\bigg)^\frac{1}{2}\\
 &\qquad +Cr\bigg(\sum_{\substack{\alpha\in\mathbb Z_+^3\\ \alpha_3\ge 3}}M_{r,\alpha+(0,2,-3)}^2\norm{D^{\alpha+(0,2,-3)}\partial_zu}_{H_x^1L_z^2}^2\bigg)^\frac{1}{2}\\
 &\qquad +C\bigg[\sum_{\substack{\alpha\in\mathbb Z_+^3\\ \alpha_3\ge 3}}\Big(t^{\frac12}r\abs{\alpha}^{-1}M_{r,\alpha}\norm{D^{\alpha}(u\partial_xu+w\partial_zu)}_{H_x^1L_z^2}\Big)^2\bigg]^\frac{1}{2}\\
&\leq Cr\abs{u}_{{Z}_r}+C\abs{u}_{{X}_r}\abs{u}_{{Z}_r},
    \end{aligned}
\end{equation*}
the last line using \eqref{hi} which follows from Hardy's inequality. Substituting the above estimate and  \eqref{est:R1} into \eqref{decom:J2} yields assertion \eqref{est:J2}.  This completes the proof of Lemma \ref{lem:J3}.
\end{proof}

\begin{lemma}[Estimate on $S_4$]\label{lem:J4}
    Let $S_4$ be given in \eqref{J1-5}, namely,  
\begin{equation*}
S_4=-\sum_{0\leq k\leq 1}\sum_{ \alpha\in\mathbb Z_+^3 }M_{r,\alpha}^2\big(D^{\alpha}\partial_z^{k+1}f|_{z=0},\ D^{\alpha} \partial_z^kf|_{z=0}\big)_{H_x^1}.  
\end{equation*}    
    It holds that
    \begin{equation}\label{est:J3}
S_4\leq Cr^\frac{1}{2}\abs{u}_{{Z}_r}^2+C\abs{u}_{{X}_r}^\frac{1}{2}\abs{u}_{{Z}_r}^2,      
    \end{equation}
where the norms $\abs{\cdot}_{{X}_r}$ and $\abs{\cdot}_{{Z}_r}$ are defined as in \eqref{dxz}.    
\end{lemma}

\begin{proof}
Using the boundary conditions $f|_{z=0}=0,$ $\partial_z^2f|_{z=0}=-\partial_xu|_{z=0}=0$ and $\partial_z^4f|_{z=0}=-\partial_x\partial_z^2u|_{z=0}=0$, we get 
\begin{multline*}
   \sum_{ \alpha\in\mathbb Z_+^3 }M_{r,\alpha}^2\inner{D^{\alpha}\partial_zf|_{z=0},\  D^{\alpha} f|_{z=0}}_{H_x^1}= \sum_{\substack{\alpha\in\mathbb Z_+^3\\ \alpha_3\ge 5}}M_{r,\alpha}^2\inner{D^{\alpha}\partial_zf|_{z=0},\  D^{\alpha} f|_{z=0}}_{H_x^1}\\
   \leq C\sum_{\substack{\alpha\in\mathbb Z_+^3\\ \alpha_3\ge 5}}M_{r,\alpha+(0,1,-2)}^2\Big|\big(D^{\alpha+(0,1,-2)}\partial_zu|_{z=0},\ D^{\alpha+(0,1,-2)}u|_{z=0}\big)_{H_x^1}\Big|.
\end{multline*}
where the last inequality follows from  \eqref{sdf} and the identity
\begin{equation*}
  D^{\alpha} f=t D^{\alpha+(0,0,-2)} \partial_z^2 f=-t D^{\alpha+(0,0,-2)} \partial_xu=-D^{\alpha+(0,1,-2)}u, 
\end{equation*}
which holds for all 
   $\alpha\in\mathbb Z_+^3\ \textrm{with}\ \alpha_3\ge 2$  
 by by the relation $\partial_xu+\partial_z^2f=0$.
Similarly,  
\begin{align*}
   &\sum_{ \alpha\in\mathbb Z_+^3}M_{r,\alpha}^2\big(D^{\alpha}\partial_z^2f|_{z=0},\ D^{\alpha} \partial_zf|_{z=0}\big)_{H_x^1}
   =\sum_{\substack{\alpha\in\mathbb Z_+^3\\ \alpha_3\ge 4}}M_{r,\alpha}^2\big(D^{\alpha}\partial_z^2f|_{z=0},\ D^{\alpha} \partial_zf|_{z=0}\big)_{H_x^1}\\
   &\leq C \sum_{\substack{\alpha\in\mathbb Z_+^3\\ \alpha_3\ge 4}}M_{r,\alpha+(0,1,-2)}^2\Big|\big(D^{\alpha+(0,1,-2)}\partial_z^2u|_{z=0},\ D^{\alpha+(0,1,-2)} \partial_zu|_{z=0}\big)_{H_x^1}\Big|.
\end{align*}
We now observe that the right-hand sides of the above inequalities correspond to boundary terms of the same type as those treated in Lemma \ref{lem:J3}. Therefore, by repeating the   proof of that lemma, we obtain the desired estimate \eqref{est:J3}. This completes the proof of Lemma \ref{lem:J4}.
\end{proof}

\begin{lemma}[Estimate on $S_5$]\label{lem:J5}
    Let $S_5$ be given in \eqref{J1-5}.
    It holds that
    \begin{equation}\label{est:J4}
S_5=\sum_{ \alpha\in\mathbb Z_+^3}\frac{2\alpha_1+2\alpha_2+\alpha_3}{2t}M_{r,\alpha}^2\norm{D^{\alpha} u}_{\mathcal{H}^1}^2\leq Cr\abs{u}_{{Z}_r}^2+C\abs{u}_{X_r}\abs{u}_{Z_r}^2,      
    \end{equation}
recalling  the norms $\abs{\cdot}_{{X}_r}$ and $\abs{\cdot}_{{Z}_r}$ are defined as in \eqref{dxz}.       
\end{lemma}

\begin{proof}
We begin by decomposing $S_5$ as 
\begin{equation}\label{s5123}
\begin{aligned}
	S_5&=\bigg(\sum_{\substack{ \alpha\in\mathbb Z_+^3\\  \alpha_3\geq 1 }} +\sum_{\substack{ \alpha\in\mathbb Z_+^3\\  \alpha_3=0,\, \alpha_2 \geq 1 }} +\sum_{\substack{ \alpha\in\mathbb Z_+^3\\  \alpha_3=\alpha_2=0,\, \alpha_1\geq 1 }}\bigg)\frac{2\alpha_1+2\alpha_2+\alpha_3}{2t}M_{r,\alpha}^2\norm{D^{\alpha} u}_{\mathcal{H}^1}^2\\
	&\stackrel{\rm def}{=}S_{5,1}+S_{5,2}+S_{5,3}. 
	\end{aligned}
\end{equation}
For $\alpha\in\mathbb Z_+^3$ with $\alpha_3\ge 1$,  recalling $\alpha_*=\alpha-(0,0,1)$ and   using \eqref{sdf} and \eqref{facttwo},  we obtain
    \begin{equation}\label{s51}
       S_{5,1}\leq Cr^2\sum_{\substack{\alpha\in\mathbb Z_+^3\\ \alpha_3\ge 1}}\frac{2\alpha_1+2\alpha_2+\alpha_3}{2\abs{\alpha}^2}M_{r,\alpha_*}^2\norm{D^{\alpha_*}\partial_z u}_{\mathcal{H}^1}^2\leq Cr^2\abs{\partial_zu}_{{X}_r}^2\leq Cr\abs{u}_{{Z}_r}^2,
    \end{equation}
   where the last inequality follows from \eqref{df} as well as $0<r<1$. 
   
   To estimate $S_{5,2}$, we claim that for $\alpha\in\mathbb Z_+^3$ with $\alpha_3=0$ and $\alpha_2\ge 1,$  the following estimate holds:
   \begin{equation}
   	\label{s52}
   \frac{1}{t}	\norm{D^{\alpha} u}_{\mathcal{H}^1}^2\leq C\norm{\comi zD^{\alpha+(0,-1,0)}\partial_xu}_{H_x^1L_z^2}\norm{\comi zD^\alpha \partial_z^2u}_{H_x^1L_z^2}.
   \end{equation}
  To verify this, fix such a multi-index $\alpha.$    Recalling 
  the definition   of   $D^\alpha$ in \eqref{def:Gamma},  and  using integration by parts and  Hardy's inequality,  we obtain  
      \begin{align*}
&	\frac{1}{t}\norm{\comi z\partial_z D^\alpha u}_{ H_x^1L_z^2}^2= \big(\comi zD^{\alpha+(0,-1,0)}\partial_x\partial_zu,\ \comi z D^\alpha \partial_zu\big)_{H_x^1L_z^2}\\
&\leq C\norm{\comi zD^{\alpha+(0,-1,0)}\partial_xu}_{H_x^1L_z^2}\big (\norm{D^\alpha \partial_zu}_{H_x^1L_z^2}+\norm{\comi zD^\alpha \partial_z^2u}_{H_x^1L_z^2}\big)\\
&\leq C\norm{\comi zD^{\alpha+(0,-1,0)}\partial_xu}_{H_x^1L_z^2}\norm{\comi zD^\alpha \partial_z^2u}_{H_x^1L_z^2},
    \end{align*} 
    and similarly,
  \begin{equation*}
  	\begin{aligned}
  	\frac{1}{t}\norm{D^\alpha u}_{ H_x^1L_z^2}^2&=\big(D^{\alpha+(0,-1,0)}\partial_xu,\ D^\alpha u\big)_{H_x^1L_z^2}\\
  	&\leq  \norm{ zD^{\alpha+(0,-1,0)}\partial_xu}_{H_x^1L_z^2} \norm{z^{-1}D^\alpha u}_{H_x^1L_z^2} \\
  	&\leq C\norm{\comi zD^{\alpha+(0,-1,0)}\partial_xu}_{H_x^1L_z^2}\norm{\comi zD^\alpha \partial_z^2u}_{H_x^1L_z^2}. 
  	\end{aligned}
  \end{equation*}
 Combining these two estimates and using the definition of the  $\mathcal H^1$-norm  in  \eqref{def:weightnorm}   yields assertion \eqref{s52}.  Moreover, it follows from \eqref{sdf} that  
 \begin{equation*}
 	\abs\alpha M_{r,\alpha} \leq  C r  M_{r, \alpha+(0,-1,0)},
 \end{equation*}     
 which along with \eqref{s52} yields
 \begin{equation}\label{s52++}
 \begin{aligned}
 	&S_{5,2}= \sum_{\substack{ \alpha\in\mathbb Z_+^3\\  \alpha_3=0,\, \alpha_2 \geq 1 }}  \frac{2\alpha_1+2\alpha_2+\alpha_3}{2t}M_{r,\alpha}^2\norm{D^{\alpha} u}_{\mathcal{H}^1}^2\\
 	&\leq  Cr \sum_{\substack{\alpha\in\mathbb Z_+^3\\  \alpha_3=0,\, \alpha_2 \geq 1}} M_{r, \alpha+(0,-1,0)} \norm{\comi zD^{\alpha+(0,-1,0)}\partial_xu}_{H_x^1L_z^2} \big(M_{r,\alpha}\norm{\comi zD^\alpha \partial_z^2u}_{H_x^1L_z^2}\big)\\
 	&\leq Cr\abs{u}_{{Z}_r}^2,
 	\end{aligned}
 \end{equation} 
 where the last inequality follows from the definition of  $\abs{\cdot}_{{Z}_r}$ in \eqref{dxz}.  

It remains to estimate $S_{5,3}$. 
Recall $\tilde\alpha=\alpha-(1,0,0)$ for $\alpha\in\mathbb Z_+^3$ with $\alpha_1\geq 1. $ Then we use the fact that
\begin{equation*}
	D^{\alpha}u=tD^{\tilde\alpha}\partial_tu\  \textrm{ and } \ \partial_tu=\partial_z^2u+\partial_xf-u\partial_xu-w\partial_zu
\end{equation*} 
 to write
\begin{equation}\label{decom:K2}
\begin{aligned}
S_{5,3}&= \sum_{\substack{ \alpha\in\mathbb Z_+^3\\  \alpha_3=\alpha_2=0,\, \alpha_1 \geq 1 }}  \frac{2\alpha_1+2\alpha_2+\alpha_3}{2t}M_{r,\alpha}^2\norm{D^{\alpha} u}_{\mathcal{H}^1}^2\\
&= \sum_{\substack{ \alpha\in\mathbb Z_+^3\\  \alpha_3=\alpha_2=0,\, \alpha_1 \geq 1 }}  \frac{2\alpha_1+2\alpha_2+\alpha_3}{2}M_{r,\alpha}^2\big(D^{\tilde \alpha}(\partial_z^2u+\partial_xf),\ D^\alpha u\big)_{\mathcal{H}^1}\\
&\qquad -\sum_{\substack{ \alpha\in\mathbb Z_+^3\\  \alpha_3=\alpha_2=0,\, \alpha_1 \geq 1 }} \frac{2\alpha_1+2\alpha_2+\alpha_3}{2}M_{r,\alpha}^2\big(D^{\tilde \alpha}(u\partial_xu+w\partial_zu),\ D^\alpha u\big)_{\mathcal{H}^1}.
\end{aligned}
\end{equation}
The condition $\alpha_3=0,$ together with the boundary conditions $u|_{z=0}=\partial_z^2 u|_{z=0}=0$,   enables us to  apply integration  by parts to get
\begin{equation}\label{uuuu}
\begin{aligned}
 &\big(D^{\tilde \alpha}\partial_z^2u,\ D^\alpha u\big)_{\mathcal{H}^1}=\big(D^{\tilde \alpha}\partial_z^2u,\ D^\alpha u\big)_{H_x^1L_z^2}+\big(\comi zD^{\tilde \alpha}\partial_z^3u,\ \comi z D^\alpha\partial_z u\big)_{H_x^1 L_z^2}\\
 &\leq \norm{D^{\tilde \alpha}\partial_zu}_{H_x^1L_z^2}\norm{D^{ \alpha}\partial_zu}_{H_x^1L_z^2}+C\norm{\comi zD^{\tilde \alpha}\partial_z^2u}_{H_x^1L_z^2}\norm{\comi zD^{ \alpha}\partial_z^2u}_{H_x^1L_z^2}\\
 &\quad +C\norm{\comi zD^{\tilde \alpha}\partial_z^2u}_{H_x^1L_z^2}\norm{D^{ \alpha}\partial_zu}_{H_x^1L_z^2}\\
 &\leq C\norm{\comi zD^{\tilde \alpha}\partial_z^2u}_{H_x^1L_z^2}\norm{\comi zD^{ \alpha}\partial_z^2u}_{H_x^1L_z^2},
\end{aligned}
\end{equation}
where the last line follows from 
Hardy's inequality.   Combining this with \eqref{facttwo} and using the definition of $\abs{\cdot}_{Z_r}$ in \eqref{dxz}, we obtain 
\begin{equation}\label{fds}
\begin{aligned}
	& \sum_{\substack{ \alpha\in\mathbb Z_+^3\\  \alpha_3=\alpha_2=0,\, \alpha_1 \geq 1 }}  \frac{2\alpha_1+2\alpha_2+\alpha_3}{2}M_{r,\alpha}^2\big(D^{\tilde \alpha} \partial_z^2u,\ D^\alpha u\big)_{\mathcal{H}^1} \\
	&\leq   Cr \sum_{  \substack{ \alpha\in\mathbb Z_+^3\\   \alpha_1 \geq 1 }} \big(  M_{r,\tilde \alpha} \norm{\comi zD^{\tilde \alpha}\partial_z^2u}_{H_x^1L_z^2}\big) \big(M_{r,\alpha}\norm{\comi zD^{ \alpha}\partial_z^2u}_{H_x^1L_z^2}\big)\leq Cr\abs{u}_{{Z}_r}^2.
	 \end{aligned}
\end{equation}
Using  the identity 
   $\partial_xu+\partial_z^2f=0$, we repeat the argument used in \eqref{uuuu} to conclude
\begin{equation*}
 \big(D^{\tilde \alpha}\partial_xf,\ D^\alpha u\big)_{\mathcal{H}^1}=\big(D^{\tilde \alpha}f,\ D^\alpha \partial_z^2 f\big)_{\mathcal{H}^1}\leq C\norm{\comi zD^{\tilde \alpha}\partial_xu}_{H_x^1L_z^2}\norm{\comi zD^{ \alpha}\partial_xu}_{H_x^1L_z^2}.
\end{equation*}
Thus,  following a similar argument as in \eqref{fds}, we have
\begin{equation*}
	\begin{aligned}
	 &\sum_{\substack{ \alpha\in\mathbb Z_+^3\\  \alpha_3=\alpha_2=0,\, \alpha_1 \geq 1 }}  \frac{2\alpha_1+2\alpha_2+\alpha_3}{2}M_{r,\alpha}^2\big(D^{\tilde \alpha} \partial_xf,\ D^\alpha u\big)_{\mathcal{H}^1} \leq Cr\abs{u}_{{Z}_r}^2.
	\end{aligned}
\end{equation*}
Combining this  with \eqref{fds}
we conclude that 
 \begin{equation*}
	\begin{aligned}
	  \sum_{\substack{ \alpha\in\mathbb Z_+^3\\  \alpha_3=\alpha_2=0,\, \alpha_1 \geq 1 }}  \frac{2\alpha_1+2\alpha_2+\alpha_3}{2}M_{r,\alpha}^2\big(D^{\tilde \alpha}(\partial_z^2u+\partial_xf),\ D^\alpha u\big)_{\mathcal{H}^1}\leq 	Cr\abs{u}_{{Z}_r}^2.	
	  \end{aligned}
\end{equation*}
On the other hand, using \eqref{facttwo} gives
\begin{multline*}
 	-\sum_{\substack{ \alpha\in\mathbb Z_+^3\\  \alpha_3=\alpha_2=0,\, \alpha_1 \geq 1 }} \frac{2\alpha_1+2\alpha_2+\alpha_3}{2}M_{r,\alpha}^2\big(D^{\tilde \alpha}(u\partial_xu+w\partial_zu),\ D^\alpha u\big)_{\mathcal{H}^1}\\
	 \leq C\sum_{\substack{ \alpha\in\mathbb Z_+^3\\   \alpha_1 \geq 1 }}  M_{r,\alpha}M_{r,\tilde \alpha} \big|\big(D^{\tilde \alpha}(u\partial_xu+w\partial_zu),\ D^\alpha u\big)_{\mathcal{H}^1}\big|\leq C\abs{u}_{X_r}\abs{u}_{Z_r}^2,
\end{multline*} 
where the last inequality follows from an analogous argument as that in Lemma \ref{j1++}.  
Combining the two estimates above with \eqref{decom:K2} yields
\begin{equation*}%\label{est:K2}
   S_{5,3}\leq Cr\abs{u}_{Z_r}^2+ C\abs{u}_{X_r}\abs{u}_{Z_r}^2.
\end{equation*}
Substituting this and \eqref{s51}, \eqref{s52++} into \eqref{s5123} yields the desired assertion \eqref{est:J4} of Lemma   \ref{lem:J5}.  
 This completes the proof.
 \end{proof}

 \begin{proof}[Completing the proof of Proposition \ref{prop:y}]
Substituting the estimates in Lemmas \ref{j1++}-\ref{lem:J5} into \eqref{est:yy} yields
\begin{equation}\label{est:prix}
  \frac{1}{2}\frac{d}{dt}\abs{u}_{{X}_r}^2+ \abs{u}_{{Z}_r}^2\leq  Cr^\frac{1}{2}\abs{u}_{{Z}_r}^2+C\big(\abs{u}_{{X}_r}^\frac{1}{2}+\abs{u}_{{X}_r}\big)\abs{u}_{{Z}_r}^2.
\end{equation}
Together with the smallness assumption \eqref{ass:initial}, this enables us to apply a standard bootstrap argument to establish the desired estimate \eqref{est:y} for sufficiently small $r$.  To do this, suppose the solution satisfies
 \begin{equation}\label{ass:pri}
\forall\ t\geq 0,\quad  \abs{u(t)}_{X_{r}}^2+\int^t_0\abs{u(s)}_{Z_{r}}^2ds\leq 2 \varepsilon_0^2.
 \end{equation}
 Then, combining \eqref{ass:pri} with \eqref{est:prix} gives
\begin{equation*}
\forall\ t\ge0,\quad \frac{1}{2}\frac{d}{dt}\abs{u}_{{X}_r}^2+ \Big(1-Cr^{\frac12}-2C\varepsilon_0^\frac{1}{2}-2C\varepsilon_0\Big)\abs{u}_{{Z}_r}^2\leq 0. 
\end{equation*}
By choosing $r, \varepsilon_0>0$ sufficiently small  such that $1-Cr^{\frac12}-2C\varepsilon_0^\frac{1}{2}-2C\varepsilon_0 \geq \frac{1}{2}$, we obtain from the above estimate  that
\begin{equation}\label{ffe}
\forall\ t\ge0,\quad \frac{d}{dt}\abs{u}_{{X}_r}^2+ \abs{u}_{{Z}_r}^2\leq 0.
\end{equation}
We now verify the short-time behavior:
\begin{equation}\label{initialtime}
     \lim_{t\to0}\abs{u(t)}_{X_r}^2= \norm{u_0}_{\mathcal{H}^1}^2.
\end{equation}
Recall that $S_5$ is defined in \eqref{J1-5}. Then 
\begin{multline*}
\int^{1}_0 t^{-1}\sum_{\substack{\alpha\in\mathbb Z_+^3\\ \abs\alpha\geq 1}} M_{r,\alpha}^2\norm{D^\alpha u}_{\mathcal{H}^1}^2dt\leq 2\int^{1}_0\sum_{\substack{\alpha\in\mathbb Z_+^3\\ \abs\alpha\geq 1}} \frac{2\alpha_1+2\alpha_2+\alpha_3}{2t}M_{r,\alpha}^2\norm{D^\alpha u}_{\mathcal{H}^1}^2dt\\
\leq C\int^{1}_0S_5dt\leq Cr\int^{1}_0\abs{u(t)}_{Z_r}^2dt+C \big(\sup_{0\leq t\leq 1}\abs{u(t)}_{X_r}\big)\int^{1}_0\abs{u(t)}_{Z_r}^2dt<+\infty,
\end{multline*}
the last line using Lemma \ref{lem:J5} and assumption \eqref{ass:pri}. This, with the continuity of the function
\begin{equation*}
    t\longrightarrow  \sum_{\substack{\alpha\in\mathbb Z_+^3\\ \abs\alpha\geq 1}} M_{r,\alpha}^2\norm{D^\alpha u}_{\mathcal{H}^1}^2,
\end{equation*}
implies that
\begin{equation*}
    \lim_{t\to 0}\sum_{\substack{\alpha\in\mathbb Z_+^3\\ \abs\alpha\geq 1}} M_{r,\alpha}^2\norm{D^\alpha u}_{\mathcal{H}^1}^2=0.
\end{equation*}
Therefore, from the definition of $\abs{u}_{X_r}$,  we deduce that
\begin{equation*}
	 \lim_{t\to 0}\abs{u}_{X_r}^2=\lim_{t\rightarrow0}  \Big(\norm{u}_{\mathcal{H}^1}^2+ \sum_{\substack{\alpha\in\mathbb Z_+^3\\ \abs\alpha\geq 1}} M_{r,\alpha}^2\norm{D^\alpha u}_{\mathcal{H}^1}^2\Big)=\norm{u_0}_{\mathcal{H}^1}^2,
\end{equation*}
which gives \eqref{initialtime}.   
Integrating \eqref{ffe} in time,  and using \eqref{initialtime} and assumption \eqref{ass:initial},  we conclude
 \begin{equation*}
   \forall\ t\geq 0,\quad  \abs{u(t)}_{X_{r}}^2+\int^t_0\abs{u(s)}_{Z_{r}}^2ds\leq   \varepsilon_0^2.
\end{equation*}
This closes the bootstrap argument and yields the desired estimate \eqref{est:y}. The proof of Proposition \ref{prop:y} is thus completed.
\end{proof}

\begin{proof}[Completing the proof of Theorem \ref{thm:2D}]
   By Proposition \ref{prop:y}, we obtain, recalling the definition of the norm $\abs{\cdot}_{{X}_r}$ in \eqref{dxz},
    \begin{equation*}
        \forall\ \alpha\in\mathbb Z_+^3,\quad  t^{\alpha_1+\alpha_2+\frac{\alpha_3}{2}}\norm{\partial_t^{\alpha_1}\partial_x^{\alpha_3}\partial_z^{\alpha_3}u}_{\mathcal{H}^1}\leq \frac{\varepsilon_0\abs{\alpha}!}{r^{\abs{\alpha}}(\abs{\alpha}+1)^4}\leq \varepsilon_0r^{-\abs{\alpha}}\abs{\alpha}!.
    \end{equation*}
Then choosing $C_0=r^{-1}$ yields assertion \eqref{result:smoothing}, which completes the proof of Theorem \ref{thm:2D}.  
\end{proof}

\section{Proof of Theorem \ref{thm:linear}}\label{sec:3D}
This section is devoted to the proof of Theorem \ref{thm:linear}, which concerns the global well-posedness  and analytic smoothing  effect of the three-dimensional linearized system \eqref{3Dlinear}.  However, as indicated in Remark \ref{anasmooth}, we will focus solely on proving the global well-posedness, since the analytic smoothing effect can be established analogously to the two-dimensional case without substantial new difficulties.

To prove the global well-posedness of system \eqref{3Dlinear}, we first recall some key estimates in the weighted Lebesgue space. With these estimates, we then establish decay properties for the coefficients $U$ and $V$ in \eqref{3Dlinear},  which enables us to conclude the global well-posedness of the system.

\subsection{Preliminaries: estimates  in the weighted Lebesgue space}
In this part, we present some estimates in the weighted Lebesgue space $L_{\mu_\lambda}^2(\mathbb R_+),$  defined as in  \eqref{leb} and  equipped with the norm
\begin{equation*}
    \norm{h}_{L_{\mu_\lambda}^2}=\Big(\int_{\mathbb{R}_+} \mu_\lambda h^2dz\Big)^{1\over2},  \quad \mu_\lambda = \exp\Big(\frac{\lambda z^2}{4(1+t)}\Big).
 \end{equation*}
Recall $\mu=\mu_1,$ that is, 
\begin{equation}\label{def:mu1}
	\mu = \exp\Big(\frac{z^2}{4(1+t)}\Big).
\end{equation}

\begin{lemma}[Lemma 2.5 in \cite{MR4701733}]\label{lema}
Let $h(t,\cdot)$ be a function belonging to $H^1_{\mu_\lambda}(\mathbb R_+)$ with $0\leq\lambda\leq 1$. Then 
\begin{align}\label{ADD1}
    \frac{\lambda^{1\over2}}{  \sqrt{2} (1+t)^{1\over2}}\norm{h}_{L_{\mu_{\lambda}}^2}\leq \norm{\partial_z h}_{L^2_{\mu_{\lambda}}},   
\end{align}
and
\begin{align*}%\label{eq2.6}
   \frac{\lambda^{1\over2}}{  2 (1+t)^{1\over2}}\norm{h}_{L_{\mu_{\lambda}}^2}+  \frac{\lambda }{4}\Big\|\frac{z}{1+t}h\Big\|_{L^2_{\mu_{\lambda}}}\leq 2\norm{\partial_z h}_{L^2_{\mu_{\lambda}}}.
\end{align*}
\end{lemma}

\begin{lemma}\label{lemv}
Let $h(t,\cdot)$ be a function belonging to $H^1_{\mu_\lambda}(\mathbb R_+)$ with $0<\lambda \leq 1$. Then  
\begin{align*}
  \norm{z^kh}_{L_z^{\infty}}\leq C_{\lambda} (1+t)^{\frac{1+2k}{4}}\norm{\partial_z h}_{L^2_{\mu_\lambda}},\quad k=0,1,2,
\end{align*}
where $C_\lambda$  is a  constant  depending  on $\lambda$.
\end{lemma}

\begin{proof} A direct computation  gives 
    \begin{align*}
    \norm{z^kh}_{L_z^{\infty}}&\leq \sup_{z\geq 0}\Big|z^k\int^{+\infty}_z\partial_zhd\tilde z\Big| \leq  \sup_{z\geq 0}\Big| \int^{+\infty}_z\tilde z^k \mu_{-\frac{\lambda}{2}}(\tilde z)\mu_{\frac{\lambda}{2}}(\tilde z)\partial_zhd\tilde z\Big|\\  
    &\leq  \norm{z^k\mu_{-\frac{\lambda}{2}}}_{L^2_z}\norm{\mu_{\frac{\lambda}{2}}\partial_zh}_{L^2_z}\leq C_{\lambda}(1+t)^{\frac{1+2k}{4}}\norm{\mu_{\frac{\lambda}{2}}\partial_zh}_{L^2_z},
    \end{align*}    
   where the last inequality uses the fact that 
\begin{equation*}\label{integr}
\forall\  \lambda<0,  \ \forall\ k\geq 0, \quad 	\norm{z^k\mu_{\lambda} }_{L_z^2}\leq C_{\lambda,k}  (1+t)^{\frac{1+2k}{4}}
\end{equation*}
with $C_{\lambda,k}$ a constant depending  on $\lambda$ and $k$. This completes the proof of  Lemma \ref{lemv}.
\end{proof}

 \begin{lemma}[Lemma 3.1 of \cite{lxy2025}]\label{lem:key}
 For any    $h(t,z)\in H_{\mu}^3$  satisfying $ h|_{z=0},$  define
\begin{align*}\label{relation}
    \mathscr{H}=\partial_zh+\frac{z}{2(1+t)}h.
\end{align*}
Then for any $0\leq \lambda<\tilde\lambda \leq 1$, the following estimate holds: 
     \begin{align*}%\label{eq2.14}
         \norm{\partial_z^{k+1}h}_{L^2_{\mu_{\lambda}}}\leq C_{\lambda,\tilde\lambda}\norm{\partial_z^k \mathscr{H}}_{L^2_{\mu_{\tilde\lambda}}},\quad k=0,1,2,
     \end{align*}
     where $C_{\lambda,\tilde{\lambda}}$  is a  constant  depending  on $\lambda$ and $\tilde\lambda$.
\end{lemma}

 \subsection{Proof of Proposition \ref{thm:heat}: decay estimate for the heat equation}\label{subheat}
 This part is devoted to proving  Proposition \ref{thm:heat}, which establishes a refined decay estimate for the coefficients $U$ and $V$ in the three-dimensional linearized system \eqref{3Dlinear}. 
Recall that  $U$ and $V$ satisfy the   following  heat equation:
\begin{equation}\label{heatzero}
    \left\{
    \begin{aligned}
  &\partial_th-\partial_z^2h=0,\\
  &h|_{t=0}=h_0,\quad h|_{z=0}=0.
    \end{aligned}
    \right.
\end{equation}
As a preliminary step toward proving Theorem \ref{thm:linear}, we derive a refined decay estimate for the heat equation \eqref{heatzero}, under the assumption that the initial data 
 $h_0$
 satisfies
\begin{align}\label{XXX}
 \int^{+\infty}_0 zh_0(z)dz=0.
\end{align}

\begin{proposition}\label{thm:prioriheat} Let $h\in L^\infty\big([0,+\infty[\ ;   H^{3}_{\mu}\big)$  satisfy the heat equation \eqref{heatzero}, where  
 the initial data $h_0\in H^3_{\mu_{in}}$  fulfills condition \eqref{XXX}.  Then for any $0<\delta<2$ it holds that   
\begin{align*}%\label{resultheat3}
\norm{\partial_zh}_{L^{\infty}_z}+\norm{z\partial_zh}_{L^{\infty}_z}+\norm{z\partial_z^2h}_{L^{\infty}_z}\leq C_\delta\norm{h_0}_{H^3_{\mu_{in}}} (1+t)^{-\frac{8-\delta}{4}}, 
\end{align*}
where $C_\delta$ is a constant depending  on $\delta.$
\end{proposition}

Note that $-\frac{8-\delta}{4}\leq -\frac32$ for $0
<\delta<2.$ Therefore,  Proposition \ref{thm:heat} is a direct consequence of Proposition \ref{thm:prioriheat}.

The proof of Proposition \ref{thm:prioriheat} is 
inspired by the works of \cite{MR4271962, MR3461362}, and   we begin by introducing some auxiliary linearly-good unknowns, following the spirit of these references. 
The first such unknown is defined as  
\begin{align}\label{defH}
      \tilde h(t,z)\stackrel{\rm def}{=} \partial_zh(t,z)+\frac{z}{2(1+t)}h(t,z).
\end{align}
This function $\tilde h$ satisfies
	\begin{equation}\label{tih}
		\left\{
		\begin{aligned}
&\Big(\partial_t-\partial_z^2+\frac{1}{1+t}\Big)\tilde h=0,\\
&\tilde h |_{t=0}=\partial_zh_0+\frac{z}{2}h_0,\quad\partial_z \tilde h |_{z=0}=0.
		\end{aligned}
		\right.
	\end{equation}	
As shown by Ignatova and Vicol \cite{MR3461362}, the additional damping term $\frac{1}{1+t}$ ensures that $\tilde h$ decays in $L^2_{\mu}$ at a rate almost like $(1+t)^{-\frac{5}{4}}$. Consequently, relation \eqref{defH} implies that $h$ itself decays at the faster rate  $(1+t)^{-\frac34}$, thereby improving upon the $L^2$ decay rate $(1+t)^{-\frac14}$ for the classical heat solution in \eqref{heatzero} with $L^1$ initial data. However, these decay rates are not fast enough to ensure the global existence.    Inspired by the work \cite{MR4271962} of Paicu and Zhang, we introduce the second  linearly-good unknown $H$ by setting
  \begin{equation}\label{defcalH}
  \begin{aligned}
     H(t,z)\stackrel{\rm def}{=}\tilde h(t,z)+\frac{z}{2(1+t)}\int^z_0 \tilde h(t,r)dr,
  \end{aligned}
  \end{equation}
where $\tilde h(t,z)$ is defined as in \eqref{defH}.
Then   $H$  satisfies  
	\begin{equation}\label{eq:mathcalH}
		\left\{
		\begin{aligned}
&\Big(\partial_t-\partial_z^2+\frac{2}{1+t}\Big)H=0,\\
&H|_{t=0}=\partial_zh_0+zh_0+\frac{z}{4}\int^{z}_0\tilde{z}h_0(\tilde z)d\tilde{z},\quad \partial_zH|_{z=0}=0.
		\end{aligned}
		\right.
	\end{equation}
Estimating $H$ in the $L^2_{\mu}$ setting requires that $H \rightarrow 0$ as $z \rightarrow +\infty$, which, by \eqref{defcalH}, is equivalent to
	\begin{equation}\label{+meva}
\forall\ t\geq 0,\quad \int_0^{+\infty} \tilde h(t,z)dz=0.
	\end{equation}
Under this condition,  the damping term $\frac{2}{1+t}$ leads to a decay rate for $H$ in $L^2_{\mu}$ that is almost like $(1+t)^{-\frac94}$ (see estimate \eqref{yed} below).  This, combined with \eqref{defcalH}, implies that $\tilde h$ also decays at the same  rate (see Lemma \ref{lem:key}), which is faster than that of system \eqref{tih}.

\begin{proof}
[Proof of Proposition \ref{thm:prioriheat}] 
 We begin by establishing the decay estimate for 
$H$, and then proceed to derive the corresponding estimate for 
$h$ which satisfies \eqref{heatzero}. The proof proceeds in two steps.

\textit{Step 1}.  Recalling $H$ is defined as in \eqref{defcalH}, we claim that for any $0< \delta<2,$
	\begin{equation}
		\label{asser}
	\sum_{j=0}^2	(1+t)^{\frac{9+2j-\delta}{4}} \norm{\partial_z^jH(t)}_{L^2_{\mu}}\leq C_{\delta}\norm{h_0}_{H^3_{\mu_{in}}}, 
	\end{equation} 
where	$C_\delta$ is a constant depending  on $\delta.$

To prove this, we first verify the validity of condition \eqref{+meva}, which enables us to estimate $\partial_z^j H$ in $L_\mu^2$ for $0\leq j\leq 2$. From the boundary condition and the assumption that $h\in L^\infty\big([0,+\infty[\ ;   H^{3}_{\mu}\big)$,  it follows that
\begin{equation}
	\label{bL}
	h|_{z=0}=h|_{z\rightarrow+\infty}=z\partial_z h|_{z\rightarrow+\infty}=0.
\end{equation}
Then we use the  identity $\partial_th=\partial_z^2h$ and  integration by parts to obtain
\begin{equation*}
	\begin{aligned}
		\partial_t \int^{+\infty}_0zh(t,z)dz=\int^{+\infty}_0z\partial_z^2h(t,z)dz=0,
	\end{aligned}
\end{equation*}
which with the assumption \eqref{XXX} yields 
\begin{equation*}\label{mv}
	\forall\ t\geq 0,\quad \int^{+\infty}_0zh(t,z)dz=0.
\end{equation*}
Moreover, this with \eqref{bL} yields, recalling $\tilde h$ is defined as in \eqref{defH},
\begin{align*}
  \forall\ t\geq 0,\quad  \int^{\infty}_0\tilde h(t, z)dz=\frac{1}{2(1+t)}\int^{+\infty}_0zh(t,z)dz=0,
\end{align*}
which gives \eqref{+meva}.

We now begin to derive the decay estimate \eqref{asser} for $H.$  By virtue of the fact that (recall $\mu$ is defined as in \eqref{def:mu1})
\begin{align*}
    \partial_t\mu =-\frac{z^2}{4(1+t)^2}\mu\quad\mathrm{and}\quad \partial_z^2\mu=\frac{1}{2(1+t)}\mu +\frac{z^2}{4(1+t)^2}\mu,
\end{align*}
we  use the conditions $\partial_z H|_{z=0}=0$ and $\partial_z\mu|_{z=0}=0$ to compute 
\begin{align*}
   2 \int_{0}^{+\infty}(\partial_t  H)H \mu dz&= \frac{d}{dt}\norm{H}^2_{L^2_{\mu}}-  \int_{0}^{+\infty} H^2 (\partial_t\mu)  dz  = \frac{d}{dt}\norm{H}^2_{L^2_{\mu}}+\frac{\norm{zH}_{L_\mu^2} ^2}{4(1+t)^2},
\end{align*}
and
\begin{align*}
  & -2\int_{0}^{+\infty}(\partial_z^2H)H \mu dz =2\norm{  \partial_z  H}^2_{L^2_{\mu}}+2\int_{0}^{+\infty}(\partial_zH)H\partial_z\mu dz \\
  &\qquad=2\norm{ \partial_zH}^2_{L^2_{\mu}} -\int_{0}^{+\infty}H^2\partial_z^2\mu dz   =2\norm{ \partial_zH}^2_{L^2_{\mu}}-\frac{\norm{H}_{L_\mu^2} ^2}{2(1+t)} -\frac{\norm{zH}_{L_\mu^2} ^2}{4(1+t)^2}.
  \end{align*}
Taking the $L_\mu^2$-product with $2H$ in \eqref{eq:mathcalH} and  combining the two identities above, we obtain\begin{equation}\label{dtvarphi}
	 \frac{d}{dt}\norm{H}^2_{L^2_{\mu}}+ 2\norm{ \partial_zH}^2_{L^2_{\mu}}+\frac{7}{2(1+t)}\norm{H}^2_{L^2_{\mu}}= 0. 
\end{equation}
Moreover,  it follows from \eqref{ADD1} in Lemma \ref{lema} that
\begin{align*}
    \frac{1}{2(1+t)}\norm{H}^2_{L^2_{\mu}}\leq \norm{ \partial_zH}^2_{L^2_{\mu}},
\end{align*}
and thus, for any $0<\delta<2$,
\begin{align*}
2\norm{ \partial_zH}^2_{L^2_{\mu}}  \geq   \delta  \norm{ \partial_zH}^2_{L^2_{\mu}}+\frac{2-\delta}{2(1+t)}\norm{H}^2_{L^2_{\mu}}.
\end{align*}\
Combining this with \eqref{dtvarphi} yields 
\begin{equation}\label{as1}
	 	\frac{d}{dt}\norm{ H}^2_{L^2_{\mu}}+\frac{9-\delta}{2(1+t)}\norm{ H}^2_{L^2_{\mu}}+\delta\norm{ \partial_z H}^2_{L^2_{\mu}}\leq 0.
\end{equation}
Noting  that $\partial_z^3H|_{z=0}=\big(\partial_t+\frac{2}{(1+t)}\big)\partial_zH|_{z=0}=0$, we  repeat the above argument to conclude that  
\begin{equation}\label{as}
	 	\frac{d}{dt}\norm{\partial_{z}^jH}^2_{L^2_{\mu}}+\frac{9-\delta}{2(1+t)}\norm{\partial_{z}^j H}^2_{L^2_{\mu}}+\delta\norm{ \partial_z^{j+1} H}^2_{L^2_{\mu}}\leq 0 \  \textrm { for } \  j=1, 2.
\end{equation}
Multiplying \eqref{as1}  by $(1+t)^{\frac{9-\delta}{2}}$ and   using the fact that 
\begin{align*}
\frac{d}{dt} (1+t)^{\frac{9-\delta}{2}} =(1+t)^{\frac{9-\delta}{2}} \frac{9-\delta}{2(1+t)},
\end{align*}
we obtain   
	\begin{equation*}
		\frac{d}{dt}\Big((1+t)^{\frac{9-\delta}{2}}\norm{H}^2_{L^2_{\mu}} \Big)+\delta (1+t)^{\frac{9-\delta}{2}}\norm{ \partial_zH}^2_{L^2_{\mu}}\leq 0.
	\end{equation*}
	Integrating the above estimate over $[0,t]$ yields 
	\begin{equation}\label{yed}
		(1+t)^{\frac{9-\delta}{2}} \norm{H(t)}^2_{L^2_{\mu}} +\delta\int_0^t(1+s)^{\frac{9-\delta}{2}}\norm{ \partial_zH(s)}^2_{L^2_{\mu}}ds\leq \norm{H(0)}^2_{L^2_{\mu_{in}}}.
	\end{equation}
For $j=1$, we rewrite the inequality \eqref{as} as 
\begin{equation*}
	\frac{d}{dt}\norm{\partial_zH}^2_{L^2_{\mu}}+\Big(\frac{11-\delta}{2(1+t)}-\frac{2}{2(1+t)}\Big)\norm{\partial_zH}^2_{L^2_{\mu}}+\delta\norm{ \partial_z^2H}^2_{L^2_{\mu}}\leq  0.
\end{equation*}
Multiplying  by $(1+t)^{\frac{11-\delta}{2}}$ yields 
\begin{equation*}
	\begin{aligned}
	  &	\frac{d}{dt}\Big((1+t)^{\frac{11-\delta}{2}}\norm{\partial_zH}^2_{L^2_{\mu}} \Big)+\delta (1+t)^{\frac{11-\delta}{2}}\norm{ \partial_z^2H}^2_{L^2_{\mu}}\leq (1+t)^{\frac{9-\delta}{2}}\norm{\partial_zH}^2_{L^2_{\mu}},
	\end{aligned}
\end{equation*}
and thus
\begin{equation}\label{y2}
	\begin{aligned}
	  &	  (1+t)^{\frac{11-\delta}{2}}\norm{\partial_zH(t)}^2_{L^2_{\mu}}  +\delta \int_0^t (1+s)^{\frac{11-\delta}{2}}\norm{ \partial_z^2H(s)}^2_{L^2_{\mu}}ds\\
	  &\leq \norm{\partial_zH(0)}^2_{L^2_{\mu}} + \int_0^t (1+s)^{\frac{9-\delta}{2}}\norm{\partial_zH(s)}^2_{L^2_{\mu}}ds\leq \norm{\partial_zH(0)}^2_{L^2_{\mu_{in}}} + \delta^{-1}\norm{H(0)}^2_{L^2_{\mu_{in}}},
	\end{aligned}
\end{equation}
where the last inequality uses \eqref{yed}. 
Similarly,  when $j=2$ we rewrite estimate \eqref{as} as 
 \begin{equation*}
 	\frac{d}{dt}\norm{\partial_z^2H}^2_{L^2_{\mu}}+\Big(\frac{13-\delta}{2(1+t)}-\frac{4}{2(1+t)}\Big)\norm{\partial_z^2H}^2_{L^2_{\mu}}+\delta\norm{ \partial_z^3H}^2_{L^2_{\mu}}\leq 0,
 \end{equation*}
 which implies 
 \begin{equation*}
	\begin{aligned}
	  &	\frac{d}{dt}\Big((1+t)^{\frac{13-\delta}{2}}\norm{\partial_z^2H}^2_{L^2_{\mu}} \Big)+\delta (1+t)^{\frac{13-\delta}{2}}\norm{ \partial_z^3H}^2_{L^2_{\mu}}\leq 2(1+t)^{\frac{11-\delta}{2}}\norm{\partial_z^2H}^2_{L^2_{\mu}}.
	\end{aligned}
\end{equation*}
 Integrating the above inequality over $[0,t]$ and using \eqref{y2}, we obtain 
 \begin{multline}\label{as3}
 	 (1+t)^{\frac{13-\delta}{2}}\norm{\partial_z^2H(t)}^2_{L^2_{\mu}}  +\delta \int_0^t (1+s)^{\frac{13-\delta}{2}}\norm{ \partial_z^3H(s)}^2_{L^2_{\mu}}ds\\
 	 \leq \norm{\partial_z^2H(0)}_{L^2_{\mu_{in}}}^2+2\delta^{-1}\norm{\partial_zH(0)}^2_{L^2_{\mu_{in}}} +2 \delta^{-2}\norm{H(0)}^2_{L^2_{\mu_{in}}}.
 \end{multline}
 Finally, using  Lemma \ref{lema} as well as the initial condition in \eqref{eq:mathcalH}, one can verify directly that  
 \begin{align*}
\sum_{j=0}^2\norm{ \partial_z^j H(0)}_{L^2_{\mu_{in}}} \leq C \norm{h_0}_{H^3_{\mu_{in}}}^2.
 \end{align*}
Combining this with estimates \eqref{as3} and  yields assertion \eqref{asser}.

\textit{Step 2}. In this step we will use \eqref{asser} to derive decay estimate for $h$. Noting \eqref{defH} and \eqref{defcalH} and  using Lemmas \ref{lemv} and \ref{lem:key} for $\lambda=\frac14$ and  $\tilde\lambda=\frac{1}{2},$ we obtain
	\begin{equation*}
		\norm{\partial_zh}_{L_z^{\infty}}\leq C(1+t)^{\frac{1}{4}}\norm{\partial_z^2h}_{L^2_{\mu_{\lambda}}}  	\leq C(1+t)^{\frac{1}{4}}\norm{\partial_z \tilde h}_{L^2_{\mu_{\tilde\lambda}}} 	\leq C(1+t)^{\frac{1}{4}} \norm{\partial_z H}_{L^2_{\mu}}.
	\end{equation*}
 Similarly, 
\begin{align*}
    \norm{z\partial_zh}_{L_z^{\infty}}\leq C(1+t)^{\frac{3}{4}}\norm{\partial_z^2h}_{L^2_{\mu_{\lambda}}}\leq C  (1+t)^{\frac{3}{4}}\norm{\partial_z H}_{L^2_{\mu}},
\end{align*}
and
\begin{align*}
    \norm{z\partial_z^2h}_{L_z^{\infty}}\leq C(1+t)^{\frac{3}{4}}\norm{\partial_z^3h}_{L^2_{\mu_{\lambda}}}\leq C(1+t)^{\frac{3}{4}}\norm{\partial_z^2 H}_{L^2_{\mu}}.
\end{align*}
Then combining these estimates and using \eqref{asser} yield that,  for any $0<\delta<2,$
\begin{equation*}
\begin{aligned}
	\norm{\partial_zh}_{L_z^{\infty}}+\norm{z\partial_zh}_{L_z^{\infty}}+\norm{z\partial_z^2h}_{L_z^{\infty}}&\leq C(1+t)^{\frac34} \big(\norm{\partial_z H}_{L^2_{\mu}}+\norm{\partial_z^2 H}_{L^2_{\mu}}\big)\\
	&\leq C_\delta(1+t)^{\frac{3}{4}-\frac{11-\delta}{4}}\leq  C_\delta(1+t)^{ -\frac{8-\delta}{4}},
	\end{aligned}
\end{equation*}
where	$C_\delta$ is a constant depending  on $\delta.$ This completes the  proof of Proposition \ref{thm:prioriheat}. 
\end{proof}

\subsection{Proof of Theorem \ref{thm:linear}}
This part is devoted to proving Theorem \ref{thm:linear} on the global well-posedness of the three-dimensional linearized system \eqref{3Dlinear}. As in the two-dimensional case, it suffices to establish the key \emph{a priori} estimate stated in  Theorem \ref{thm:apriori2} below.

Let $(u,v)$ be any given solution to the initial-boundary problem \eqref{3Dlinear}. To simplify the notation, we denote
\begin{align*}
    \vec{a}=(u,v).
\end{align*}
Let $|\vec{a}|_{\mathcal{X}_{\rho}}$ be given as in Definition \ref{defspace}, namely,
\begin{equation}\label{defX}
	 |\vec{a}|^2_{\mathcal{X}_{\rho}}=\sum^1_{j=0}\,\sum_{m=0}^{+\infty}L_{\rho,m}^2\|\partial_{y}^m\partial_z^j \vec{a}\|^2_{L^2}.
\end{equation}
We define $|\vec{a}|_{\mathcal{Y}_{\rho}}$ and $|\vec{a}|_{\mathcal{Z}_{\rho}}$ by setting
\begin{align}\label{defY}
 |\vec{a}|^2_{\mathcal{Y}_{\rho}}\stackrel{\rm def}{=}\sum^1_{j=0}\,\sum_{m=0}^{+\infty}\frac{m+1}{\rho}L_{\rho,m}^2\|\partial_{y}^m\partial_z^j \vec{a}\|^2_{L^2},   
\end{align}
and
\begin{align}\label{defZ}
|\vec{a}|^2_{\mathcal{Z}_{\rho}}\stackrel{\rm def}{=}\sum^1_{j=0}\,\sum_{m=0}^{+\infty}L_{\rho,m}^2\|\partial_{y}^m\partial_z^{j+1} \vec{a}\|^2_{L^2}+\sum_{m=0}^{+\infty}L_{\rho,m}^2\|\partial_x\partial_{y}^m\vec{a}\|^2_{L^2}.   
\end{align}
Here $L_{\rho,m}$ and $\rho$ are defined by \eqref{lrm} and \eqref{3Ddefrho}, respectively.

\begin{theorem}[\emph{A priori} estimate]\label{thm:apriori2}
Suppose that the hypothesis of  Theorem \ref{thm:linear} holds. Let $(u,v)\in L^{\infty}([0,+\infty[;\mathcal{X}_{\rho})$ be a solution to the system \eqref{3Dlinear}, satisfying that
\begin{align*}
    \forall\ t>0,\quad \abs{\vec{a}(t)}_{\mathcal{X}_{\rho}}^2+\int^t_0\abs{\vec{a}(s)}_{\mathcal{Z}_{\rho}}^2ds<+\infty,
\end{align*}
where
  $|\vec{a}|_{\mathcal{X}_{\rho}}$ and $|\vec{a}|_{\mathcal{Y}_{\rho}}$   are defined in  \eqref{defX} and \eqref{defZ}, respectively. Then
\begin{align*}
    \forall\ t>0,\quad \abs{\vec{a}(t)}_{\mathcal{X}_{\rho}}^2+\int^t_0\abs{\vec{a}(s)}_{\mathcal{Z}_{\rho}}^2ds\leq \abs{\vec{a}(0)}_{\mathcal{X}_{\rho_0}}^2.
\end{align*}

\end{theorem}

\begin{proof} 
Applying $\partial_y^m\partial_z$ to the first  and second  equations in \eqref{3Dlinear} yields
\begin{multline*}
	(\partial_t +U\partial_x + V\partial_y - \partial_z^2)\partial_y^m\partial_zu\\
	=\partial_x \partial_y^m\partial_zf-(\partial_y^mw)\partial_z^2U+(\partial_zU)\partial_y^{m+1}v-(\partial_zV)\partial_y^{m+1}u, 
\end{multline*}
and 
\begin{multline*}
	 (\partial_t +U\partial_x + V\partial_y - \partial_z^2)\partial_y^m\partial_zv\\
	 =\partial_x \partial_y^m\partial_zg-(\partial_y^mw)\partial_z^2V+(\partial_zV)\partial_x\partial_y^{m}u-(\partial_zU)\partial_x\partial_y^{m}v.
\end{multline*}
 We perform energy estimate for these  equations and use the boundary condition $\partial_z^2u|_{z=0}=\partial_z^2v|_{z=0}=0$  and  the identities:
\begin{equation*}
	 \big(\partial_x\partial_y^m\partial_zf,\ \partial_y^m\partial_zu\big)_{L^2}= \big(\partial_y^m\partial_z^2 f,\ \partial_x\partial_y^mu\big)_{L^2}=-\norm{\partial_x\partial_y^mu}_{L^2}^2,
\end{equation*}
and 
\begin{equation*}
	 \big(\partial_x\partial_y^m\partial_zg,\ \partial_y^m\partial_zv\big)_{L^2}= \big(\partial_y^m\partial_z^2g,\ \partial_x\partial_y^mv\big)_{L^2}=-\norm{\partial_x\partial_y^mv}_{L^2}^2,
\end{equation*}
which follow from the fact that $\partial_xu+\partial_z^2f=0$ and $\partial_xv+\partial_z^2g=0$; 
this gives
\begin{equation*}
	\begin{aligned}
&\frac{1}{2}\frac{d}{dt}\|\partial_{y}^m\partial_z\vec{a}\|^2_{L^2}+ \big(\|\partial_{y}^m\partial_z^2\vec{a}\|^2_{L^2}+\|\partial_x\partial_{y}^m\vec{a}\|^2_{L^2}\big)\\
&\leq  C\norm{(z\partial_z^2V,z\partial_z^2U)}_{L^{\infty}} \norm{z^{-1}\partial_y^mw}_{L^2}\norm{\partial_y^m\partial_z\vec{a}}_{L^2}\\
&\quad+C\norm{(\partial_zU,\partial_zV)}_{L^{\infty}} \big(\norm{\partial_y^{m+1}\vec{a}}_{L^2}+\norm{\partial_x\partial_y^{m}\vec{a}}_{L^2}\big)\norm{\partial_y^m\partial_z\vec{a}}_{L^2}\\
&\leq C\sum_{1\leq j\leq 2} \norm{z^{j-1}\partial_z^{j}(U, V)}_{L^{\infty}} \big(\norm{\partial_y^{m+1}\vec{a}}_{L^2}+\norm{\partial_x\partial_y^{m}\vec{a}}_{L^2}\big)\norm{\partial_y^m\partial_z\vec{a}}_{L^2}\\
&\leq C\eps_1(1+t)^{-\frac32} \big(\norm{\partial_y^{m+1}\vec{a}}_{L^2}+\norm{\partial_x\partial_y^{m}\vec{a}}_{L^2}\big)\norm{\partial_y^m\partial_z\vec{a}}_{L^2}.
\end{aligned}
\end{equation*}
where the last lines uses Proposition   \ref{thm:heat}, and  the second inequality  follows from  the estimate 
\begin{equation*}
 \norm{ z^{-1} \partial_y^mw}_{L^2}\leq C\norm{\partial_x \partial_y^mu}_{L^2}+C\norm{ \partial_y^{m+1}v}_{L^2},  
\end{equation*}
which is a consequence of Hardy's  inequality.  Multiplying both sides by $L^2_{\rho,m}$ and then summing over $m\in\mathbb Z_+$, we use the identity
\begin{equation*}\label{eqrh}
\forall m\ge0,\quad \frac{1}{2}\frac{d}{dt}L^2_{\rho,m}=\rho'\frac{m+1}{\rho}L^2_{\rho,m}
\end{equation*}
to derive 
\begin{equation}\label{GB}
\begin{aligned}
&\frac{1}{2}\frac{d}{dt}\sum^{+\infty}_{m=0}L_{\rho,m}^2\|\partial_{y}^m\partial_z\vec{a}\|^2_{L^2}+\sum^{+\infty}_{m=0}L_{\rho,m}^2\big(\|\partial_{y}^m\partial_z^2\vec{a}\|^2_{L^2}+\|\partial_x\partial_{y}^m\vec{a}\|^2_{L^2}\big)\\
&\leq \rho'\sum^{+\infty}_{m=0}\frac{m+1}{\rho}L_{\rho,m}^2\|\partial_{y}^m\partial_z\vec{a}\|^2_{L^2} \\
&\quad+C\eps_1(1+t)^{-\frac{3}{2}} \sum^{+\infty}_{m=0}L_{\rho,m}^2\big(\norm{\partial_y^{m+1}\vec{a}}_{L^2}+\norm{\partial_x\partial_y^{m}\vec{a}}_{L^2}\big)\norm{\partial_y^m\partial_z\vec{a}}_{L^2}.
\end{aligned}
\end{equation}
For the last summation in \eqref{GB},  recalling  $|\vec{a}|_{\mathcal{Y}_{\rho}}$ and $|\vec{a}|_{\mathcal{Z}_{\rho}}$ are defined as in \eqref{defY} and \eqref{defZ}, respectively,  we compute
\begin{equation*}
	 \begin{aligned}
&\sum^{+\infty}_{m=0}L_{\rho,m}^2\big(\norm{\partial_y^{m+1}\vec{a}}_{L^2}+\norm{\partial_x\partial_y^{m}\vec{a}}_{L^2}\big)\norm{\partial_y^m\partial_z\vec{a}}_{L^2}\\
&\leq   \bigg (\sum^{+\infty}_{m=0}\frac{\rho}{m+1}L_{\rho,m}^2  \norm{ \partial_y^{m+1}\vec{a}}_{L^2}^2 \bigg)^{\frac12}\bigg(\sum^{+\infty}_{m=0}\frac{m+1}{\rho}L_{\rho,m}^2  \norm{\partial_y^m\partial_z\vec{a}}_{L^2}^2\bigg)^{\frac12}\\
&\quad + \bigg (\sum^{+\infty}_{m=0} L_{\rho,m}^2  \norm{\partial_x\partial_y^m\vec{a}}_{L^2}^2 \bigg)^{\frac12}\bigg(\sum^{+\infty}_{m=0} L_{\rho,m}^2  \norm{\partial_y^m\partial_z\vec{a}}_{L^2}^2\bigg)^{\frac12}\\
&\leq C\bigg(\sum^{\infty}_{m=0} \frac{m+2}{\rho}L^2_{\rho,m+1}\norm{ \partial_y^{m+1}\vec{a}}_{L^2}^2 \bigg)^{\frac12} |\vec a|_{\mathcal Y_\rho}+ |\vec{a}|_{\mathcal{Z}_{\rho}}|\vec{a}|_{\mathcal{X}_{\rho}}\\
&\leq C |\vec{a}|_{\mathcal{X}_{\rho}} |\vec{a}|_{\mathcal{Z}_{\rho}}+C|\vec{a}|_{\mathcal{Y}_{\rho}}^2,
    \end{aligned}
\end{equation*}
where the second inequality uses the fact that 
\begin{equation*}
	\frac{\rho}{m+1}L_{\rho,m}^2 \leq C \frac{m+2}{\rho} L_{\rho,m+1}^2. 
\end{equation*}
Combining these estimates yields 
\begin{equation}\label{zu}
\begin{aligned}
& \frac{1}{2}\frac{d}{dt}\sum^{+\infty}_{m=0}L_{\rho,m}^2\|\partial_{y}^m\partial_z\vec{a}\|^2_{L^2}+\sum^{+\infty}_{m=0}L_{\rho,m}^2\big(\|\partial_{y}^m\partial_z^2\vec{a}\|^2_{L^2}+\|\partial_x\partial_{y}^m\vec{a}\|^2_{L^2}\big)\\
& \leq \rho'\sum^{+\infty}_{m=0}\frac{m+1}{\rho}L_{\rho,m}^2\|\partial_{y}^m\partial_z\vec{a}\|^2_{L^2}  +C\eps_1(1+t)^{-\frac{3}{2}} \Big( |\vec{a}|_{\mathcal{X}_{\rho}}|\vec{a}|_{\mathcal{Z}_{\rho}}+ |\vec{a}|_{\mathcal{Y}_{\rho}}^2\Big).
\end{aligned}
\end{equation}
On the other hand, applying an analogous argument to the equations 
\begin{equation*}
		\left\{
		\begin{aligned}
            &  (\partial_t +U\partial_x + V\partial_y - \partial_z^2)\partial_y^mu=\partial_x \partial_y^mf-(\partial_y^mw)\partial_zU,\\
            &  (\partial_t +U\partial_x + V\partial_y - \partial_z^2)\partial_y^mv=\partial_x \partial_y^mg-(\partial_y^mw)\partial_zV,
		\end{aligned}
		\right.
	\end{equation*}
and using Proposition   \ref{thm:heat} again, we obtain 
\begin{equation}\label{uv}
\begin{aligned}
& \frac{1}{2}\frac{d}{dt}\sum^{+\infty}_{m=0}L_{\rho,m}^2\|\partial_{y}^m \vec{a}\|^2_{L^2}+\sum^{+\infty}_{m=0}L_{\rho,m}^2 \|\partial_{y}^m\partial_z\vec{a}\|^2_{L^2}+\sum^{+\infty}_{m=0}L_{\rho,m}^2\|\partial_{y}^m\partial_z(f,g)\|^2_{L^2} \\
& \leq \rho'\sum^{+\infty}_{m=0}\frac{m+1}{\rho}L_{\rho,m}^2\|\partial_{y}^m \vec{a}\|^2_{L^2}  +C\eps_1(1+t)^{-\frac{3}{2}} \Big( |\vec{a}|_{\mathcal{X}_{\rho}}|\vec{a}|_{\mathcal{Z}_{\rho}}+ |\vec{a}|_{\mathcal{Y}_{\rho}}^2\Big),
\end{aligned}
\end{equation}
where the third term on the left side arises from 
 the identities:
\begin{equation*}
	\big(\partial_x\partial_y^mf,\partial_y^mu\big)_{L^2}= -\norm{\partial_y^m\partial_zf}_{L^2}^2\ \textrm{ and }\ \big(\partial_x\partial_y^mg,\partial_y^mv\big)_{L^2}=-\norm{\partial_y^m\partial_z g}_{L^2}^2.
\end{equation*}
By definitions of $|\vec{a}|_{\mathcal{X}_{\rho}}$, $|\vec{a}|_{\mathcal{Y}_{\rho}}$ and $|\vec{a}|_{\mathcal{Z}_{\rho}}$,  we combine \eqref{zu} and \eqref{uv} to conclude that
\begin{multline*}
 \frac{1}{2}\frac{d}{dt}|\vec{a}|_{\mathcal{X}_{\rho}}^2-\rho'|\vec{a}|_{\mathcal{Y}_{\rho}}^2+|\vec{a}|_{\mathcal{Z}_{\rho}}^2 
     \leq C\varepsilon_1(1+t)^{-\frac{3}{2}} \Big( |\vec{a}|_{\mathcal{X}_{\rho}}|\vec{a}|_{\mathcal{Z}_{\rho}}+ |\vec{a}|_{\mathcal{Y}_{\rho}}^2\Big)\\
\leq \frac{1}{2}|\vec{a}|_{\mathcal{Z}_{\rho}}^2+C\varepsilon_1^2(1+t)^{-3}|\vec{a}|_{\mathcal{X}_{\rho}}^2+ C\eps_1(1+t)^{-\frac{3}{2}}  |\vec{a}|_{\mathcal{Y}_{\rho}}^2,
\end{multline*}
which with the fact that 
$ |\vec{a}|_{\mathcal{X}_{\rho}}^2\leq \rho |\vec{a}|_{\mathcal{Y}_{\rho}}^2$ 
yields 
\begin{align*}
   & \frac{1}{2}\frac{d}{dt}|\vec{a}|_{\mathcal{X}_{\rho}}^2-\rho'|\vec{a}|_{\mathcal{Y}_{\rho}}^2+\frac12|\vec{a}|_{\mathcal{Z}_{\rho}}^2 
     \leq    C\eps_1(1+\eps_1)(1+t)^{-\frac{3}{2}}  |\vec{a}|_{\mathcal{Y}_{\rho}}^2.
\end{align*}
From the definition \eqref{3Ddefrho} of $\rho,$  it follows that 
\begin{equation*}
	\rho'=-\frac{\rho_0}{4}(1+t)^{-\frac32}.
\end{equation*}
Then
\begin{align*}
    \frac{1}{2}\frac{d}{dt}|\vec{a}|_{\mathcal{X}_{\rho}}^2 +\frac12|\vec{a}|_{\mathcal{Z}_{\rho}}^2\leq \Big[-\frac{\rho_0}{4}+  C \eps_1(1+\eps_1)\Big](1+t)^{-\frac32} |\vec{a}|_{\mathcal{Y}_{\rho}}^2\leq 0,  
\end{align*}
provided $\eps_1$ is chosen sufficiently small. 
This yields
\begin{align*}
    \forall\ t\ge0,\quad \abs{\vec{a}(t)}_{\mathcal{X}_{\rho}}^2+\int^t_0\abs{\vec{a}(s)}_{\mathcal{Z}_{\rho}}^2ds\leq \abs{\vec{a}(0)}_{\mathcal{X}_{\rho_0}}^2.
\end{align*}
The proof of Theorem \ref{thm:apriori2} is thus completed.
\end{proof}

\appendix

\section{Proofs of inequalities \eqref{ineq3}, \eqref{ineq4},  \eqref{ineq5} and \eqref{ineq6}} \label{sec:ineq}
 In this part we present the proofs of several inequalities; the arguments are straightforward. In what follows, we let $\alpha=(\alpha_1,\alpha_2,\alpha_3)\in\mathbb Z_+^3$ and $\beta=(\beta_1,\beta_2,\beta_3)\in\mathbb Z_+^3$ be any  multi-indices, satisfying $\alpha_3\ge 1$ and $\beta\leq\alpha$.

 \begin{proof}[Proofs of \eqref{ineq3} and \eqref{ineq4}]
For $\beta_3=0$, we have $\alpha-\beta+(0,1,-1)\in\mathbb Z_+^3$. Then we use the fact \eqref{factor} to compute, recalling $M_{r,\alpha}$ is defined in \eqref{def:Hthetam},
 \begin{align*}
     &\binom{\alpha}{\beta}\frac{\abs{\alpha}^{-1}M_{r,\alpha}}{M_{r,\beta}M_{r,\alpha-\beta+(0,1,-1)}}\\
     &\leq \frac{\abs{\alpha}!}{\abs{\beta}!(\abs{\alpha}-\abs{\beta)}!}\frac{r^{\abs{\alpha}}(\abs{\alpha}+1)^4\abs{\alpha}^{-1}}{\abs{\alpha}!}\frac{\abs{\beta}!}{r^{\abs{\beta}}(\abs{\beta}+1)^4}\frac{(\abs{\alpha}-\abs{\beta})!}{r^{\abs{\alpha}-\abs{\beta}}(\abs{\alpha}-\abs{\beta}+1)^4}\\
     &\leq \frac{(\abs{\alpha}+1)^4\abs{\alpha}^{-1}}{(\abs{\beta}+1)^4(\abs{\alpha}-\abs{\beta}+1)^4}\leq \frac{(\abs{\alpha}+1)^4}{(\abs{\beta}+1)^4(\abs{\alpha}-\abs{\beta}+1)^4}\\
     &\leq \frac{C}{(\abs{\beta}+1)^4}+\frac{C}{(\abs{\alpha}-\abs{\beta}+1)^4},
 \end{align*}
 the last inequality following from that
 \begin{equation*}
    \frac{(\abs{\alpha}+1)^4}{(\abs{\beta}+1)^4(\abs{\alpha}-\abs{\beta}+1)^4}\leq \frac{C}{(\abs{\beta}+1)^4}\quad \textrm{if}\quad 0\leq\abs{\beta}\leq \Big[\frac{\abs{\alpha}}{2}\Big], 
 \end{equation*}
 and
 \begin{equation*}
    \frac{(\abs{\alpha}+1)^4}{(\abs{\beta}+1)^4(\abs{\alpha}-\abs{\beta}+1)^4}\leq \frac{C}{(\abs{\alpha}-\abs{\beta}+1)^4}\quad \textrm{if}\quad \Big[\frac{\abs{\alpha}}{2}\Big]+1\leq\abs{\beta}\leq\abs{\alpha}. 
 \end{equation*}
 Here we denote by $[p]$ the largest integer less than or equal
 to $p$.
 
For $\beta_3\ge 1$, we recall $\beta_*=\beta-1=(\beta_1,\beta_2,\beta_3-1)\in\mathbb Z_+^3$ and $\abs{\beta_*}=\abs{\beta}-1$.  A similar computation applied to \eqref{ineq3} yields
  \begin{align*}
     &\binom{\alpha}{\beta}\frac{\abs{\alpha}^{-1}M_{r,\alpha}}{M_{r,\beta_*}M_{r,\alpha-\beta+(0,1,0)}}\\
     &\leq \frac{\abs{\alpha}!}{\abs{\beta}!(\abs{\alpha}-\abs{\beta})!}\frac{r^{\abs{\alpha}}(\abs{\alpha}+1)^4\abs{\alpha}^{-1}}{\abs{\alpha}!}\frac{(\abs{\beta}-1)!}{r^{\abs{\beta}-1}\abs{\beta}^4}\frac{(\abs{\alpha}-\abs{\beta}+1)!}{r^{\abs{\alpha}-\abs{\beta}+1}(\abs{\alpha}-\abs{\beta}+2)^4}\\
     &\leq \frac{\abs{\alpha}+1)^4\abs{\alpha}^{-1}(\abs{\alpha}-\abs{\beta}+1)}{\abs{\beta}^5(\abs{\alpha}-\abs{\beta}+2)^4}\leq \frac{\abs{\alpha}+1)^4}{\abs{\beta}^4(\abs{\alpha}-\abs{\beta}+2)^4}\\
     &\leq \frac{C}{\abs{\beta}^4}+\frac{C}{(\abs{\alpha}-\abs{\beta}+2)^4}.
 \end{align*}
The proofs of \eqref{ineq3} and \eqref{ineq4} are thus completed.
 \end{proof}

 \begin{proof}[Proofs of \eqref{ineq5} and \eqref{ineq6}]
 The proofs of inequalities \eqref{ineq5} and \eqref{ineq6} are quite analogous to those of inequalities \eqref{ineq3} and \eqref{ineq4}. We compute that
   \begin{align*}
       &\binom{\alpha}{\beta}\frac{r \abs{\alpha}^{-1}M_{r,\alpha}}{M_{r,\beta}M_{r,\alpha-\beta+(0,0,1)}}\\
       &\leq \frac{\abs{\alpha}!}{\abs{\beta}!(\abs{\alpha}-\abs{\beta)}!}\frac{r^{\abs{\alpha}+1}(\abs{\alpha}+1)^4\abs{\alpha}^{-1}}{\abs{\alpha}!}\frac{\abs{\beta}!}{r^{\abs{\beta}}(\abs{\beta}+1)^4}\frac{(\abs{\alpha}-\abs{\beta}+1)!}{r^{\abs{\alpha}-\abs{\beta}+1}(\abs{\alpha}-\abs{\beta}+2)^4}\\
     &\leq \frac{(\abs{\alpha}+1)^4\abs{\alpha}^{-1}(\abs{\alpha}-\abs{\beta}+1)}{(\abs{\beta}+1)^4(\abs{\alpha}-\abs{\beta}+2)^4}\leq \frac{(\abs{\alpha}+1)^4}{(\abs{\beta}+1)^4(\abs{\alpha}-\abs{\beta}+2)^4}\\
     &\leq \frac{C}{(\abs{\beta}+1)^4}+\frac{C}{(\abs{\alpha}-\abs{\beta}+2)^4}
   \end{align*}  
 for $\beta_3=0$,  and
   \begin{align*}
     &\binom{\alpha}{\beta}\frac{\abs{\alpha}^{-1}M_{r,\alpha}}{M_{r, \beta+(0,1,-1)}M_{r,\alpha-\beta}}\\
     &\leq \frac{\abs{\alpha}!}{\abs{\beta}!(\abs{\alpha}-\abs{\beta})!}\frac{r^{\abs{\alpha}}(\abs{\alpha}+1)^4\abs{\alpha}^{-1}}{\abs{\alpha}!}\frac{\abs{\beta}!}{r^{\abs{\beta}}(\abs{\beta}+1)^4}\frac{(\abs{\alpha}-\abs{\beta})!}{r^{\abs{\alpha}-\abs{\beta}}(\abs{\alpha}-\abs{\beta}+1)^4}\\
        &\leq \frac{(\abs{\alpha}+1)^4\abs{\alpha}^{-1}}{(\abs{\beta}+1)^4(\abs{\alpha}-\abs{\beta}+1)^4}\leq \frac{(\abs{\alpha}+1)^4}{(\abs{\beta}+1)^4(\abs{\alpha}-\abs{\beta}+1)^4}\\
     &\leq \frac{C}{(\abs{\beta}+1)^4}+\frac{C}{(\abs{\alpha}-\abs{\beta}+1)^4}
 \end{align*}
 for $\beta_3\ge 1$. The proofs are thus completed.
 \end{proof}

%	\bibliographystyle{abbrv}
	
%	\bibliography{references}

\end{document}